\documentclass[12pt]{article}

\usepackage{enumitem}
\usepackage{mathtools}
\usepackage{xr}

\usepackage{url}
\usepackage{amsfonts}
\usepackage{amsmath,amssymb,color}
\usepackage{graphicx}
\usepackage{enumitem}
\usepackage{multirow}
\usepackage{bm}
\usepackage{verbatim}
\usepackage{float}
\usepackage[toc,page]{appendix}
\usepackage[doublespacing]{setspace}
\usepackage{apacite}
\usepackage{natbib}
\usepackage{adjustbox}
\usepackage{rotating}
\usepackage{subcaption}
\usepackage{amsthm}
\usepackage{booktabs,dcolumn,caption}

\usepackage{mathabx}

\usepackage{array}

\usepackage{amsthm}
\usepackage{cleveref}

\numberwithin{equation}{section}
\numberwithin{figure}{section}

\newcolumntype{L}[1]{>{\raggedright\let\newline\\\arraybackslash\hspace{0pt}}p{#1}}
\newcolumntype{C}[1]{>{\centering\let\newline\\\arraybackslash\hspace{0pt}}p{#1}}
\newcolumntype{R}[1]{>{\raggedleft\let\newline\\\arraybackslash\hspace{0pt}}p{#1}}

\newcommand{\pr}{{\text{pr}}}
\newcommand{\bfs}{{\boldsymbol s}}
\def\T{{ \mathrm{\scriptscriptstyle T} }}

\theoremstyle{plain}
\newtheorem{theorem}{Theorem}[section]
\newtheorem{lemma}[theorem]{Lemma}

\newtheorem{case}{Case}

\newtheorem{corollary}{Corollary}

\newtheorem{definition}{Definition}

\title{\LARGE Adjusting for non-confounding covariates in \\case-control association studies}
\author{Siliang Zhang, Jinbo Chen, Zhiliang Ying and Hong Zhang}

\begin{document}
\maketitle

\doublespacing

\begin{abstract}
There is a considerable literature in case-control logistic regression on whether or not non-confounding covariates should be adjusted for. However, only limited and ad hoc theoretical results are available on this important topic. A constrained maximum likelihood method was recently proposed, which appears to be generally more powerful than logistic regression methods with or without adjusting for non-confounding covariates. This note provides a theoretical clarification for the case-control logistic regression with and without covariate adjustment and the constrained maximum likelihood method on their relative performances in terms of asymptotic relative efficiencies. We show that the benefit of covariate adjustment in the case-control logistic regression depends on the disease prevalence. We also show that the constrained maximum likelihood estimator gives an asymptotically uniformly most powerful test. 
\end{abstract}	
\noindent
KEY WORDS: Asymptotic relative efficiency; Case-control design; Constrained maximum likelihood; Logistic regression; Non-confounding covariate.

\section{Introduction}

It is well known that adjusting for baseline covariates can lead to an improved statistical inference efficiency \citep{fisher1932}. 
Through rigorous derivation, \citet{robinson1991some} showed that in logistic regression, adjusting for non-confounding covariates in randomized clinical trials always benefits the testing for treatment effect in terms of Pitman's asymptotic relative efficiency, albeit with some estimation precision loss. 
\citet{neuhaus1998} extended the results of \citet{robinson1991some} to generalized linear models.

It is more complicated, however, when data are collected retrospectively but analyzed with prospective logistic regression. In particular, \citet{kuo2010} showed through simulations that adjusting for non-confounding covariates in case-control studies may decrease the statistical inference efficiency for exposure-disease association; see \citet{xing2010adjusting} for additional comments.
\citet{pirinen2012} showed that adjusting for non-confounding covariates in case-control studies can decrease estimation precision, resulting in a loss of power provided that the estimation is approximately unbiased. In general, covariate adjustment could result in both bias and loss of efficiency in estimation, and it is not clear how they affect the hypothesis testing for exposure-disease association.

\cite{zaitlen2012informed, zaitlen2012analysis} proposed a liability threshold model that exploits covariate-specific prevalence information to improve the inference efficiency of exposure-disease association. \citet{mpMLE2017} developed a constrained profile maximum likelihood method with known disease prevalence and independence between the exposure and covariate. 
Their simulation results indicate that the method outperforms the standard logistic regression with or without adjusting for covariates. 

In this paper, we derive theoretical properties for the case-control logistic regression methods with and without covariate adjustment, and the constrained maximum likelihood method. Specifically, when both exposure and covariate are dichotomous, we derive the asymptotic biases, asymptotic variances and asymptotic relative efficiencies of the three methods. Furthermore, we obtain the asymptotic distribution of the constrained maximum likelihood estimator with a possibly misspecified disease prevalence. Our theoretical findings are at least threefold. First, adjusting for non-confounding covariates can decrease power in case-control studies with a low disease prevalence, which extends the finding of \cite{robinson1991some} and \cite{neuhaus1998} to case-control studies.
Second, the constrained maximum likelihood method has a uniform power advantage over the other two methods. Third, the constrained maximum likelihood method is robust against prevalence misspecification. 

\section{Models and Approaches}\label{model&method}

Consider a case-control study design involving a binary response variable $D$ ($D=1$: case; $D=0$: control), a binary exposure variable of interest $E$ ($E=1$: exposed; $E=0$: unexposed) and a binary covariate $X$ ($X=1$: high risk category; $X=0$: low risk category). The exposure variable $E$ could be a genetic mutation or an environmental exposure. The covariate $X$ is assumed to be independent of $E$ throughout this paper. Note that spurious association could be produced when $X$ and $E$ are dependent but $X$ is not adjusted for.
Let $f=\pr(D=1)$, $\theta=\pr(X=1)$ and $\pi=\pr(E=1)$ be the prevalences of $D$, $X$ and $E$, respectively, in the population from which cases and controls are sampled. Throughout this paper, we assume that $X$ and $E$ are non-degenerate so that $0<\pi,\theta<1.$
Furthermore, we assume that the following logistic regression model holds true:
\begin{equation}\label{Adj}
p_{ij}(\alpha,\beta,\gamma)=\pr(D=1\mid X=i,E=j)=\frac{\exp(\alpha+\beta i+\gamma j)}{1+\exp(\alpha+\beta i+\gamma j)},
\end{equation}
where $\alpha$ is the baseline log-relative risk, and $\beta$ and $\gamma$ are log odds ratios. Note that the above model does not involve $E$-$X$ interaction term due to the assumption that the $E$-$D$ odds ratio does not depend on $X$. We are interested in testing the null hypothesis of no association between $D$ and $E$ with the adjustment of $X$, i.e., $H_0: \gamma=0$. Under the case-control study design, data for $(E,X)$ are randomly sampled from case population ($D=1$) and control population ($D=0$). Let $n_{dij}$ denote the number of subjects with $D=d$, $X=i$ and $E=j$. The total numbers of cases and controls are denoted by $n_{1++}$ and $n_{0++}$, respectively, and let $\nu=n_{1++}/n_{0++}$. 

This paper provides a theoretical clarification on the asymptotic relative efficiencies for three methods. The first method, henceforth referred to as ``\textsc{Adj}", fits model (\ref{Adj}) to the case-control data by adjusting for $X$ as if the data were prospectively collected. The corresponding estimator of $\gamma$, $\hat\gamma_A$, is the maximizer of the prospective likelihood function, which is consistent, asymptotically normally distributed, and semiparametric efficient \citep{anderson1972,prentice1979logistic, breslow2000}. \citet{robinson1991some} derived a closed-form estimator of the asymptotic variance of $\hat\gamma_A$. The null hypothesis, $H_0:\gamma=0$, can be tested using a Wald statistic. 

The second method, referred to as ``\textsc{Mar}", tests the marginal association between $D$ and $E$. \textsc{Mar} fits the following logistic regression without adjusting for $X$:
\begin{equation}
\label{Mar}
\pr(D=1\mid E=j) = \frac{\exp(\alpha_0 + \gamma_0 j)}{1+\exp(\alpha_0 + \gamma_0 j)},
\end{equation}
where $\alpha_0$ is the marginal baseline log-relative risk and $\gamma_0$ is the marginal $E$-$D$ log odds ratio. Note that $\alpha_0$ and $\gamma_0$ generally differ from $\alpha$ and $\gamma$ in model (\ref{Adj}) unless $\beta$ equals zero. The corresponding maximum likelihood estimator of $\gamma_0$, denoted by $\hat\gamma_M$, takes the form $\hat\gamma_M = \log(n_{1+1}/n_{1+0}) - \log(n_{0+1}/n_{0+0}).$
The null hypothesis of no association between $D$ and $E$, formulated as $\gamma_0=0$, can again be tested using a Wald statistic. 

The third method, referred to as ``\textsc{AdjCon}", is based on a constrained maximum likelihood method \citep{mpMLE2017}. \textsc{AdjCon} maximizes the same retrospective likelihood function of the first method by additionally imposing two constraints. The first constraint is that $E$ and $X$ are independent, and the second constraint is that the true prevalence of $D$ is known. The second constraint can be expressed as 
$f = \sum_{i=0}^1\sum_{j=0}^1\pr(D=1\mid X=i,E=j)\pr(X=i)\pr(E=j)$, or equivalently,
\begin{equation}\label{theta}
\theta = \{f-p_{01}\pi-p_{00}(1-\pi)\}/\{p_{11}\pi+p_{10}(1-\pi)-p_{01}\pi-p_{00}(1-\pi)\},
\end{equation}
where $p_{ij}=p_{ij}(\alpha,\beta,\gamma)$ is defined in \eqref{Adj}.
The retrospective likelihood function under these two constraints $f=\pr(D=1)$ can be obtained by profiling out $\theta$:
\begin{equation}\label{likfun}
\prod_{d=0}^1\prod_{i=0}^1\prod_{j=0}^1 \{\pr(D=d\mid X=i,E=j)\pr(X=i)\pr(E=j)\}^{n_{dij}},
\end{equation}
where $\pr(X=1)$ is replaced with the right-hand side of (\ref{theta}). The corresponding maximum likelihood estimator, denoted by $\hat\gamma_{AC}$, can then be numerically obtained using any non-linear optimization algorithm. The null hypothesis $H_0:\gamma=0$ can be tested using a Wald statistic based on $\hat\gamma_{AC}$. Intuitively, \textsc{AdjCon} should be more efficient than \textsc{Adj} since the former makes use of two additional constraints.
Indeed, simulation results in \cite{mpMLE2017} showed that \textsc{AdjCon} outperforms both \textsc{Adj} and \textsc{Mar}.

\section{Main Results}\label{mainresult}
In this section, we establish theoretical properties for the estimators $\hat\gamma_M$, $\hat\gamma_A$ and $\hat\gamma_{AC}$. Specifically, asymptotic biases are derived in Section \ref{sec:asymptotic_bias}; asymptotic distributions are presented in Section \ref{sec:asymptotic_dist}; asymptotic relative efficiencies for testing the null hypothesis are given in Section \ref{sec:asymptotic_re}; robustness of $\hat\gamma_{AC}$ when prevalence is misspecified is investigated in Section \ref{sec:misspecification}.
\subsection{Asymptotic bias without covariate adjustment}\label{sec:asymptotic_bias}
Here we derive an asymptotic expression and the corresponding asymptotic bias for $\hat{\gamma}_M$. Let $n=n_{1++} + n_{0++}$. To avoid singularity, assume that $\nu = n_{1++}/n_{0++}$ is bounded away from zero and infinity.

\begin{lemma}\label{lemma1}
We have the following asymptotic expansion for the marginal maximum likelihood estimator $\hat{\gamma}_M$:
	\begin{equation}\label{egamma0}
	\hat{\gamma}_M=\gamma+\delta+O_P(n^{-1/2}) \text{ as }n \to \infty,
	\end{equation}
	where 
	\begin{equation}\label{delta}
		\delta = \log\bigg\{ 1 + \frac{e^\alpha(b_1 - b_2)(1-e^\gamma)}{(1+e^{\alpha+\gamma} b_1)(1+e^\alpha b_2)}\bigg\},
	\end{equation}
and $b_1=1+(e^\beta-1)(1-\theta), b_2=\{1+(e^{-\beta}-1)(1-\theta)\}^{-1}.$
\end{lemma}

It can be shown that $b_1\geq b_2$, with the equality holding if and only if $\gamma=0$ or $\beta=0$. Furthermore, the signs of $\gamma$ and $\delta$ are opposite and $|\delta|\le|\gamma|$. Refer to (S6) and (S7) in Supplementary Material for details. As a result, we have the following corollary.
\begin{corollary}\label{corollary1}
The limiting value of $\hat\gamma_M$, $\gamma+\delta$, is shrunk towards zero (i.e., $\vert\gamma+\delta\vert\leq \vert\gamma\vert$ and the signs of $\gamma$ and $\delta$ are opposite). Furthermore, the asymptotic bias $\delta$ equals zero if and only if $X$ or $E$ is not associated with $D$. Finally, $|\delta|$ is maximized at $f=f^*$:
\begin{equation}\label{fmin}
f^*=\sum_{i=0}^1\sum_{j=0}^1 p_{ij}(\alpha^*,\beta,\gamma) \theta^i(1-\theta)^{1-i}\pi^j(1-\pi)^{1-j},
\end{equation}
where
\begin{equation}\label{alphamin}
\alpha^*= -\frac{1}{2}\{\log(b_1b_2) + \gamma\}.
\end{equation}
\end{corollary}

Lemma \ref{lemma1} and Corollary \ref{corollary1} confirm the empirical observations that the maximum likelihood estimator of $\gamma$ is conservative when ignoring non-confounding covariates \cite[]{stringer2011}. The asymptotic unbiasedness conditions $\beta=0$ and $\gamma=0$ correspond to two non-confounding assumptions in prospective logistic regression \citep{robinson1991some}. \citet{gail1984biased} and \citet{neuhausjewell1993} obtained similar results, but only for $\beta$ around zero. 
In contrast, our results hold for general $\beta$. Moreover,  $f\to 0$ implies $\delta\to 0$, i.e., adjusting for $X$ results in a very small bias in the low disease prevalence situation, which is consistent with previous findings \citep{lee1982specification}. Figure \ref{figure12}(A) displays the asymptotic bias in one parameter setting based on Lemma \ref{lemma1}. The bias appears to increase with $f$ for $f\in (0,f^*]$ and decrease with $f$ for $f\in [f^*,1)$.
Unlike $\hat\gamma_M$, both $\hat{\gamma}_A$ and $\hat{\gamma}_{AC}$ are asymptotically unbiased under model \eqref{Adj}. That is, $\hat{\gamma}_A = \gamma + O_P(n^{-1/2})$ \citep{anderson1972,prentice1979logistic} and $\hat{\gamma}_{AC} = \gamma + O_P(n^{-1/2})$
\citep{mpMLE2017}.

\begin{figure}
	\centering
	\includegraphics[width=1.0\linewidth]{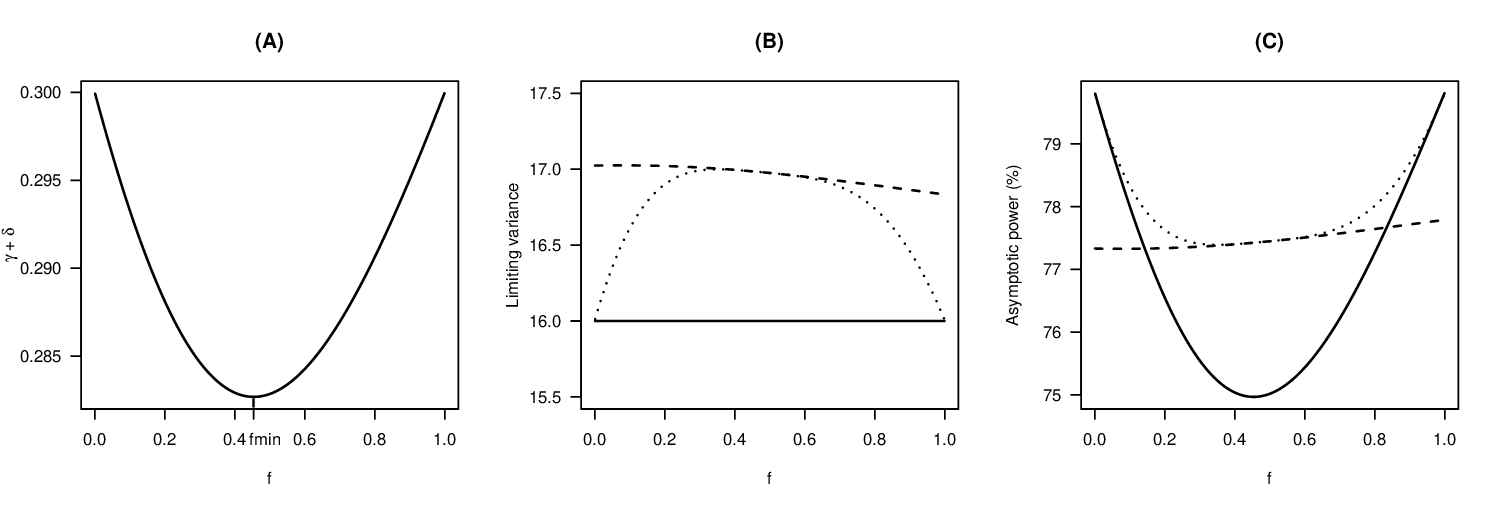}
	\caption{(A) The limiting value of $\hat{\gamma}_M$, $\gamma+\delta$, as a function of the disease prevalence $f$, 
	with the underlying parameters being $\pi=0.5$, $\theta=0.4$, $\beta=1$ and $\gamma=0.3$; (B) The limiting variances of $n^{1/2}\hat{\gamma}_A$ (dashed line), $n^{1/2}\hat{\gamma}_{AC}$ (dotted line) and $n^{1/2}\hat{\gamma}_M$ (solid line), with the underlying parameters being $\pi=0.5$, $\theta=0.4$, $\beta=1$, $\gamma=0.05$ and $\nu=1$; (C) The asymptotic powers of \textsc{Mar} (solid line), \textsc{Adj} (dashed line) and \textsc{AdjCon} (dotted line), with the underlying parameters being $n=5\times10^4$, $q=1$, $\theta=0.4$, $\pi=0.5$, $\beta=1$ and $\gamma=0.05$.}
	\label{figure12}
\end{figure}

\subsection{Asymptotic normality}\label{sec:asymptotic_dist}

In this subsection, we establish the asymptotic normality for $\hat \gamma_M$, $\hat \gamma_A$ and $\hat \gamma_{AC}$.
\begin{lemma}\label{lemma2}
As $n$ goes to infinity, $n^{1/2}(\hat{\gamma}_M-\gamma-\delta)$, $n^{1/2}(\hat{\gamma}_A-\gamma)$ and $n^{1/2}(\hat{\gamma}_{AC}-\gamma)$ converge in distribution to normal with mean zero and variances $\sigma_M^2, \sigma_A^2$ and $\sigma_{AC}^2$, whose explicit expressions are given in (S8), (S10) and (S14) in Supplementary Material.
\end{lemma}

We can analytically compare $\sigma_M^2, \sigma_A^2$ and $\sigma_{AC}^2$ using their explicit expressions, especially when $\gamma\to0$ (and $f\to0$), as detailed in the following corollary.

\begin{corollary}\label{corollary2}
For $\sigma_M^2$, $\sigma_{AC}^2$ and $\sigma_A^2$, we have the following:

1. $\sigma_M^2 \leq \sigma_A^2$, with equality holding if and only if $\beta=0$; 

2. If $\gamma\to0$, then $\sigma^2_{M}\rightarrow\sigma^2_0$ and $\sigma_A^2\rightarrow \lambda\sigma_0^2$, where $\sigma_0^2=(2+\nu+1/\nu)/\{\pi(1-\pi)\}$, 
\begin{equation}\label{lambda}
\lambda=1 + \frac{\nu\theta(1-\theta)}{(1+\nu)}\frac{(1-e^\beta)^2}{ \big\{(1-\theta)\phi + e^\beta\theta\phi^{-1}\big\}^2+\nu e^\beta \big\{ (1-\theta)\phi + \theta\phi^{-1} \big\}^2} \text{ with } \phi=\sqrt{\frac{1+e^{\alpha+\beta}}{1+e^{\alpha}}},
\end{equation}
where $\lambda\geq 1$ and the equality holds if and only if $\beta=0$;

3. If $\gamma\rightarrow 0$ and $f\rightarrow 0$, then we have $\lambda\rightarrow \lambda_0$, $\sigma^2_{AC}\rightarrow\sigma^2_0$ and $\sigma_A^2\rightarrow\lambda_0  \sigma_0^2$, where
\begin{equation}\label{lambda0}
\lambda_0 = 1+\frac{\nu\theta(1-\theta)}{(1+\nu)}\frac{(1-e^\beta)^2}{(1-\theta+e^\beta\theta)^2+\nu e^\beta}\ge1,
\end{equation}
and $\lambda_0=1$ if and only if $\beta=0$.
\end{corollary}

Unlike linear models, adjusting for non-confounding covariates in case-control logistic regression always leads to an increase in the variance of $\gamma$ estimator, i.e., $\sigma_M^2\leq\sigma_A^2$, as claimed in Colollary \ref{corollary2}. This result agrees with the finding of \cite{robinson1991some} for prospective studies.
In the rare disease situation, Corollary \ref{corollary2} states that $\hat{\gamma}_M$ and $\hat{\gamma}_{AC}$ have the same asymptotic variance, and that $\hat{\gamma}_A$ has a larger asymptotic variance unless the covariate $X$ is independent of the disease $D$. 
This result complements that of \cite{pirinen2012}, which only derives the approximated ratio of the variances for $\hat{\gamma}_M$ and $\hat{\gamma}_A$.
Figure \ref{figure12}(B) displays the asymptotic variances as functions of $f$ in one parameter setting. The asymptotic variance of $\hat{\gamma}_M$ appears to be the smallest in general. The variance of $\hat{\gamma}_A$ is the largest among the three estimators. On the other hand, the variance of $\hat{\gamma}_{AC}$ falls in between the other two, and it first increases then decreases with $f$.

We now compare the Wald test performances of the three methods (\textsc{Mar}, \textsc{Adj} and \textsc{AdjCon}) under the contiguous alternative hypothesis $H_1:\gamma=c_1n^{-1/2}$, where $c_1$ is a fixed non-zero constant. With the asymptotic expectation and variance for each method in Lemmas 1 and 2, we can derive the limiting power for the corresponding Wald statistic under $H_1$. 
As shown in Figure \ref{figure12}(C), \textsc{AdjCon} appears to be more powerful than \textsc{Adj}, which is due to the asymptotic unbiasedness of $\hat{\gamma}_A$ and $\hat{\gamma}_{AC}$ and the smaller asymptotic variance of \textsc{AdjCon}, especially when $f$ is close to 0 or 1. When $f$ approaches 0.5, the power gain of \textsc{AdjCon} over \textsc{Adj} becomes negligible, as their asymptotic variances converge.
When $f$ is close to 0 or 1, \textsc{Mar} appears to be more powerful than \textsc{Adj}, as they have similar means but \textsc{Adj} gives a larger variance. However, \textsc{Mar} becomes less powerful than the other two methods as $f$ takes value around 0.5. This stems from the fact that $\hat{\gamma}_M$ is considerably biased toward zero and its variance advantage cannot be compensated for the bias disadvantage.
\textsc{AdjCon} appears to be slightly more powerful than \textsc{Mar} when $f$ is close to 0 or 1 (Figure \ref{figure12}(C)). The theoretical results in the next subsection confirm a part of the above empirical findings.

\subsection{Asymptotic relative efficiencies}\label{sec:asymptotic_re}

Asymptotic power comparison of various methods is carried out analytically through Pitman's asymptotic relative efficiency \citep{pitman1979, serfling2009}.
For test statistics $T_1$ and $T_2$, the $T_1$ vs. $T_2$ Pitman asymptotic relative efficiency is equal to $e_P(T_1,T_2)=\lim (m_2/m_1)$ (refer to Supplementary Material for details), where $m_1$ and $m_2$ are the sample sizes of $T_1$ and $T_2$ for achieving the same asymptotic power under the contiguous alternative hypotheses $\gamma=cm_2^{-1/2}$ and $\gamma=cm_1^{-1/2}$, respectively. Therefore, $e_P(T_1,T_2) > 1$ indicates that $T_1$ is asymptotically more powerful than $T_2$, and vice versa.
Denote by $T_M$, $T_A$ and $T_{AC}$ the Wald statistics corresponding to $\hat{\gamma}_M$, $\hat{\gamma}_A$ and $\hat{\gamma}_{AC}$, respectively. Let $\rho = e^{\alpha}$, which is related to the disease prevalence. We evaluate $e_P(T_M,T_A)$ and $e_P(T_M,T_{AC})$ in Theorem \ref{theorem1} and Theorem \ref{theorem2}, respectively.

\begin{theorem}	\label{theorem1}
	The $T_M$ vs. $T_A$ Pitman asymptotic relative efficiency has the following asymptotic representation:
	\begin{equation}\label{PARE0}
	e_P(T_M, T_A)=\bigg\{\frac{b_1b_2\rho^2+2b_2\rho+1}{b_1b_2\rho^2+(b_1+b_2)\rho+1}\bigg\}^2 \lambda,
	\end{equation}
	where $\lambda\geq 1$ is defined in \eqref{lambda} and $b_1$ and $b_2$ are defined below \eqref{delta}.
\end{theorem}

Theorem \ref{theorem1} gives an analytical form for $e_P(T_M, T_A)$, which allows us to evaluate their asymptotic relative efficiencies under different prevalence levels. In particular, we have the following result in the rare disease situation.

\begin{corollary}\label{corollary3}
	For rare disease (i.e., $\rho\to0$), the $T_M$ vs. $T_A$ Pitman asymptotic relative efficiency has the following asymptotic expansion:
	\begin{equation}\label{PARE1}
	e_P(T_M, T_A) = \lambda_0 + O(\rho),
	\end{equation}
	where $\lambda_0\ge1$ with $\lambda_0$ being defined in \eqref{lambda0}, and the equality holds if and only  $X$ and $D$ are independent (or equivalently $\beta = 0$).
\end{corollary}

It is not surprising that \textsc{AdjCon} is generally more powerful than \textsc{Adj} since the two methods are based on the same model but the former incorporates two additional constraints. 
Furthermore, as indicated in Figure \ref{figure12}(C), \textsc{AdjCon} also appears to be more powerful than \textsc{Mar}. The following theorem gives a theoretical justification.

\begin{theorem}	\label{theorem2}
	For rare disease (i.e., $\rho\to0$), the $T_M$ vs. $T_{AC}$ Pitman asymptotic relative efficiency has the following asymptotic representation:
	\begin{equation}\label{PARE2}
	e_P(T_M, T_{AC}) = 1 + \tau\rho^2 + o(\rho^2),
	\end{equation}
	where $\tau\leq 0$ with $\tau$ being defined in  (S28) in Supplementary Material, and the equality holds if and only if $X$ and $D$ are independent (or equivalently $\beta=0$).
\end{theorem}

\subsection{Constrained maximum method under prevalence misspecification}\label{sec:misspecification}

This section studies robustness of \textsc{AdjCon} against misspecification of disease prevalence. Numerical studies of \cite{mpMLE2017} suggested that \textsc{AdjCon} is not very sensitive to the misspecification. Before stating our theoretical result, we introduce some more notations and assumptions. 
Let $\bfs = (\beta,\gamma,\theta,\pi)^{\T}$ denote unknown model parameters. Note that the intercept parameter $\alpha$ is determined by $f$ and $\bfs$ according to \eqref{theta}. Denote by $l_f(\cdot)$ the log-likelihood function with given prevalence $f$. 
Let the true prevalence be $f_0$. Let $\bfs_f^*$ denote the maximizer of $E_{f_0}\{l_f(\bfs)\}$. Let $\hat\bfs_{f}$ denote the maximum likelihood estimator of $\bfs_f$ with the disease prevalence being specified to be $f$.
Throughout this section, we assume that $f$ is bounded away from 1 (i.e., $f\in(0,1-\epsilon]$ for some $\epsilon>0$) and for all $(\beta^*,\gamma^*,\theta^*,\pi^*)\in\{\bfs_f^*,f\in(0,1-\epsilon]\}$, $\beta^*$ and $\gamma^*$ are bounded away from infinity and $\theta^*$ and $\pi^*$ are bounded away from zero and one.

\begin{theorem}	\label{theorem3}
For any specified prevalence $f\in (0,1-\epsilon]$, we have the following asymptotic properties:
\begin{equation}
    \sqrt{n}(\hat\bfs_{f} - \bfs_{f}^*) \rightarrow N(0, \Sigma_{f}(\bfs_{f}^*))\text{ in distribution},
\end{equation}
\begin{equation}\label{eq:quasi_deviation}
    \Vert\bfs_{f}^* - \bfs_{f_0}^*\Vert\leq C_1\vert f-f_0\vert,
\end{equation}
and 
\begin{equation}
  \Vert\Sigma_{f}(\bfs_{f}^*) - \Sigma_{f_0}(\bfs_{f_0}^*)\Vert \leq C_2\vert f-f_0\vert,
\end{equation}
where $\Vert\cdot\Vert$ is the Euclidean norm, $\Sigma_f(\bfs_{f}^*)$ is the asymptotic covariance matrix of $\hat\bfs_{f}$ evaluated at $\bfs_{f}^*$, and $C_1$ and $C_2$ are constants independent of $f$.
\end{theorem}

Theorem \ref{theorem3} establishes the asymptotic normality of the maximum likelihood estimator with a possibly misspecified disease prevalence. Furthermore, the limiting value $\bfs_{f}^*$ of the maximum likelihood estimator $\hat\bfs_{f}$ and the corresponding asymptotic covariance matrix $\Sigma_{f}(\bfs_{f}^*)$ are Lipschitz continuous with respect to $f$, indicating that the statistical inference is not very sensitive to disease prevalence misspecification. 

\section*{Acknowledgement}\label{acknoledgement}
The research was supported in part by the Natural Science Foundation of China (12171451, 72091212), the US National Institutes of Health (R01-CA236468, R01-HL138306), the Shanghai Science and Technology Committee Rising-Star Program (22YF1411100).

\setcounter{equation}{0}
\renewcommand{\theequation}{A\arabic{equation}}

\appendix
\section*{Appendix}

The proofs of Lemma 1 and Corollary 1 are presented in Section~\ref{sec:proof_lemma1} and Section~\ref{sec:proof_coro1}, respectively. The proofs of Lemma 2 and Corollary 2 are given in Section~\ref{sec:proof_lemma2_coro2}. The definition of Pitman's asymptotic relative efficiency (Section 3.3 of the main text) is restated in Section~\ref{sec:define_pitman}. The proofs of Theorem 1 and Corollary 3 are presented in Section~\ref{sec:proof_theorem1_coro3}. The proof of Theorem 2 is presented in Section~\ref{sec:proof_theorem2}. 
The proof of Theorem 3 is given in Section~\ref{sec:proof_theorem3}. 
Additional discussion is provided in Section~\ref{sec:additional_discussion}.

\section{Proof of Lemma 1}\label{sec:proof_lemma1}
We adopt the notations of the main text, for example,
$$
  f = \pr(D=1),\quad \theta = \pr(X=1),\quad \pi = \pr(E=1).
$$
We also introduce some additional notations:
\begin{equation}\label{notation}
\begin{split}
  p_i = \pr(D=1\mid E=i)&,\quad q_{i} = 1-p_{i} = \pr(D=0\mid E=i),\\
  p_{ij} = \pr(D=1\mid X=i,E=j)&,\quad q_{ij} =1-p_{ij} = \pr(D=0\mid X=i,E=j),
\end{split}
\end{equation}
for $i=0,1;\ j=0,1$.

Throughout this document, we assume that $X$ and $E$ are independent unless specially noted, so that $p_i = p_{1i}\theta + p_{0i}(1-\theta),\ q_{i}= q_{1i}\theta + q_{0i}(1-\theta)$. Under the retrospective setting, the random variables $n_{1+1}$ and $n_{0+1}$ follow binomial distributions, i.e., $n_{1+1} \sim B(n_{1++}, p_1')$ and $n_{0+1}\sim B(n_{0++}, p_0')$, where
\begin{align}
p_1'= \pr(E=1\mid D=1) = \frac{p_1\pi}{f},&\quad p_0' = \pr(E=1\mid D=0) = \frac{q_1\pi}{1-f},\\
q_1'= \pr(E=0\mid D=1) = \frac{p_0(1-\pi)}{f},&\quad q_0'=\pr(E=0\mid D=0) = \frac{q_0(1-\pi)}{1-f}.\label{q1'}
\end{align} 

For any $f\in(0,1)$, it follows from the standard large sample theory for the sample odds ratio and (\ref{notation})-(\ref{q1'}) that
\begin{align}
\hat \gamma_M &= \log \frac{p_1'q_0'}{p_0'q_1'} + O_P(n^{-\frac{1}{2}})\nonumber \\
&= \log\frac{p_1q_0}{p_0q_1} + O_P(n^{-\frac{1}{2}})\nonumber\\
&= \log\frac{\{p_{11}\theta+p_{01}(1-\theta)\}\{q_{10}\theta+q_{00}(1-\theta)\}}{\left(p_{10}\theta+p_{00}(1-\theta)\right)\left(q_{11}\theta+q_{01}(1-\theta)\right)}+ O_P(n^{-\frac{1}{2}})\nonumber\\
&=\gamma + \log\bigg\{\frac{(1-\theta+e^{\alpha+\beta+\gamma}+e^\beta\theta)(1+e^\alpha\theta+e^{\alpha+\beta}(1-\theta))}{(1-\theta+e^{\alpha+\beta}+e^\beta\theta)(1+e^{\alpha+\gamma}\theta+e^{\alpha+\beta+\gamma}(1-\theta))}\bigg\} + O_P(n^{-\frac{1}{2}})\nonumber\\
&= \gamma + \log\bigg\{1 + \frac{e^\alpha(b_1 - b_2)(1-e^\gamma)}{(1+e^\alpha b_2)(1 + e^{\alpha+\gamma} b_1)}\bigg\}+O_P(n^{-\frac{1}{2}})\nonumber\\  
&= \gamma+\delta+O_P(n^{-\frac{1}{2}}), \label{gammahat}
\end{align}
where 
$$  b_1 = e^\beta(1-\theta) + \theta=1+(e^\beta-1)(1-\theta),\ b_2 = e^\beta/(1-\theta + e^\beta\theta)=\frac{1}{1+(e^{-\beta}-1)(1-\theta)},
$$
$$\hat\gamma_M=\log(n_{1+1}/n_{1+0}) - \log(n_{0+1}/n_{0+0}),$$ 
and
\begin{equation}\label{delta_sup}
    \delta = \log\bigg\{ 1 + \frac{e^\alpha(b_1 - b_2)(1-e^\gamma)}{(1+e^{\alpha+\gamma} b_1)(1+e^\alpha b_2)}\bigg\}.
\end{equation}
Moreover,
\begin{equation}\label{eq:b1b2_ineq}
    b_1 = 1 + (e^\beta-1)(1-\theta)= \frac{1+(e^\beta+e^{-\beta}-2)\theta(1-\theta)}{1+(e^{-\beta}-1)(1-\theta)}\geq \frac{1}{1+(e^{-\beta}-1)(1-\theta)}=b_2>0.
\end{equation}

\section{Proof of Corollary 1}\label{sec:proof_coro1}
It follows from $b_1\geq b_2 >0$ that
\begin{equation}\label{eq:b1b2_ineq2}
-\gamma < \delta \leq 0 \mbox{ if } \gamma>0 \mbox{ and } 0 \leq \delta < -\gamma \mbox{ if } \gamma<0 \mbox{ and } \delta=0 \mbox{ if } \gamma = 0,
\end{equation}
so that
$$
|\gamma+\delta| \leq |\gamma|.
$$
Furthermore, it is easily seen from the expression of $\delta$ given in (\ref{delta_sup}) of the main text that $\delta=0$ if and only if $b_1=b_2$ (which leads to $\beta=0$, i.e., $X$ is not associated with $D$) or $\gamma=0$ (i.e., $E$ is not associated with $D$).
Finally, setting the derivative of (\ref{delta_sup}) with respect to $\alpha$ to be 0, we can see that $|\delta|$ is minimized at $\alpha_{\min}$ defined in (8) of the main text.

\section{Proof of Lemma 2 and Corollary 2}\label{sec:proof_lemma2_coro2}

As defined in the main text, $\nu=n_{1++}/n_{0++}$, so that $n_{0++}=n/(1+\nu)$ and $n_{1++}=n\nu/(1+\nu)$. Assume a contiguous alternative scenario where $\gamma=cn^{-1/2}$.

First, we derive the asymptotic distribution of $\hat\gamma_M$. According to the standard large sample theory, the regularity conditions R1 - R3 \citep[see Chapter 4,][]{serfling2009} hold for logistic regression models, which gives
$$
n^{1/2} (\hat\gamma_M-\gamma-\delta)\rightarrow N(0,\sigma^2_M)\text{ in distribution},
$$
where the asymptotic variance is
\begin{equation}\label{sigmaM}
\sigma^2_M = \frac{n}{n_{0++}p_0'q_0'}+ \frac{n}{n_{1++}p_{1}'q_{1}'}=\frac{(1+\nu)}{p_0'q_0'} +  \frac{(1+\nu)}{\nu p_1'q_1'}.
\end{equation}
Since $\gamma=0$ implies that $p_1'=p_0'=\pi$, we have that
\begin{equation}\label{sigmaM0}
\sigma^2_M\rightarrow\sigma^2_0 \mbox{ as } \gamma\rightarrow0,
\end{equation}
where $\sigma_0^2=(2+\nu+1/\nu)/\{\pi(1-\pi)\}$. The above results hold for any $f\in(0,1)$.

Next we derive the asymptotic distribution of $\hat\gamma_A$. According to \cite{gart1962combination}, we have
$$
  n^{1/2}\left(\hat \gamma_A - \gamma\right) \rightarrow N(0,\sigma_A^2)\text{ in distribution},
$$
where
\begin{equation}\label{sigma_A}
\begin{aligned}
  \sigma_A^2 = & \bigg\{\left(\frac{n}{n_{0++}d_{00}h_{00}(1-h_{00})}+\frac{n}{n_{1++}d_{01}h_{01}(1-h_{01})}\right)^{-1}+\\
  &\left(\frac{n}{n_{0++}d_{10}h_{10}(1-h_{10})}+\frac{n}{n_{1++}d_{11}h_{11}(1-h_{11})}\right)^{-1}\bigg\}^{-1}\\
  =& \bigg\{\left(\frac{1+\nu}{d_{00}h_{00}(1-h_{00})} + \frac{1+\nu}{\nu d_{01}h_{01}(1-h_{01})}\right)^{-1}+\\
  & \left(\frac{1+\nu}{d_{10}h_{10}(1-h_{10})}+\frac{1+\nu}{\nu d_{11}h_{11}(1-h_{11})}\right)^{-1}\bigg\}^{-1}
\end{aligned}
\end{equation}
and
$$  d_{ij} = \pr(X=i\mid D=j),\quad h_{ij}=\pr(E=1\mid X=i,D=j).$$

If we denote
\begin{align*}
  a_{10} = \frac{d_{00}h_{00}}{1+\nu},\quad a_{20} = \frac{d_{00}(1-h_{00})}{1+\nu},\quad a_{30} = \frac{\nu d_{01}h_{01}}{1+\nu},\quad a_{40} = \frac{\nu d_{01}(1-h_{01})}{1+\nu},\\
  a_{11} = \frac{d_{10}h_{10}}{1+\nu},\quad a_{21} = \frac{d_{10}(1-h_{10})}{1+\nu},\quad a_{31} = \frac{\nu d_{11}h_{11}}{1+\nu},\quad a_{41} = \frac{\nu d_{11}(1-h_{11})}{1+\nu},
\end{align*}
then we have
$$  \sigma_M^2 = \frac{1}{a_{10}+a_{11}}+\frac{1}{a_{20}+a_{21}}+\frac{1}{a_{30}+a_{31}}+\frac{1}{a_{40}+a_{41}}$$
and
$$  \sigma_A^2 =  \bigg\{\left(\frac{1}{a_{10}}+\frac{1}{a_{20}}+\frac{1}{a_{30}}+\frac{1}{a_{40}}\right)^{-1}+\left(\frac{1}{a_{11}}+\frac{1}{a_{21}}+\frac{1}{a_{31}}+\frac{1}{a_{41}}\right)^{-1}\bigg\}^{-1}.
$$
Applying the Minkowski inequality, we immediately have that $\sigma_M^2 \leq \sigma_A^2$, and the inequality holds even when $X$ and $E$ are correlated. Moreover, the equality holds if and only if $a_{i1} = k a_{i0}\ (i=1,\ldots,4)$, or equivalently, $X$ is independent of $D$ (i.e., $\beta=0$).

Next, we compare $\sigma_M^2$ and $\sigma_A^2$ under the condition of $\gamma \rightarrow 0$. We rewrite the asymptotic variances as
\begin{align*}
  \sigma_M^2 =& \frac{(1+\nu)(1-f)^2}{q_0 q_1\pi(1-\pi)}+\frac{(1+\nu)f^2}{\nu p_0 p_1\pi(1-\pi)}\nonumber\\
  =& \frac{1}{\pi(1-\pi)} \bigg\{\frac{(1+\nu)(1-f)^2}{q_0 q_1} + \frac{(1+\nu)f^2}{\nu p_0 p_1}\bigg\}\nonumber\\
  =& \frac{1+\nu}{\pi(1-\pi)} \bigg\{\frac{(1-f)^2}{E(q_{X0}) E(q_{X1})} + \frac{f^2}{\nu E(p_{X0}) E(p_{X1}) }\bigg\}
\end{align*}
and
\begin{align*}
\sigma_A^2 =& \bigg\{(1-\theta)\left(\frac{(1+\nu)(1-f)}{q_{01}\pi}+\frac{(1+\nu)(1-f)}{q_{00}(1-\pi)}+\frac{(1+\nu)f}{\nu p_{01}\pi}+\frac{(1+\nu)f}{\nu p_{00}(1-\pi)}\right)^{-1} + \nonumber\\
  &\ \theta\left(\frac{(1+\nu)(1-f)}{q_{11}\pi}+\frac{(1+\nu)(1-f)}{q_{10}(1-\pi)}+\frac{(1+\nu)f}{\nu p_{11}\pi}+\frac{(1+\nu)f}{\nu p_{10}(1-\pi)}\right)^{-1}\bigg\}^{-1}\nonumber\\
  =& \left(1+\nu\right)\left[ E \bigg\{\frac{(1-f)}{q_{X1}\pi}+\frac{(1-f)}{q_{X0}(1-\pi)}+\frac{f}{\nu p_{X1}\pi}+\frac{f}{\nu p_{X0}(1-\pi)}\bigg\}^{-1}\right]^{-1}.
\end{align*}

If $\gamma\rightarrow 0$, then $p_{i1}\rightarrow p_{i0}$ and $q_{i1}\rightarrow q_{i0}$ for $i=1,2.$ Consequently,
\begin{align}
  &\lim_{\gamma\rightarrow 0} \frac{\sigma_A^2}{\sigma_M^2} 
  = \frac{\left[E \bigg\{\frac{(1-f)}{q_{X0}}+\frac{f}{\nu p_{X0}}\bigg\}^{-1}\right]^{-1}}{\left(\frac{1-f}{Eq_{X0}}\right)^2+\frac{1}{\nu}\left(\frac{f}{Ep_{X0}}\right)^2}\nonumber\\
  =& \frac{(1+\frac{1}{\nu})\left[ \bigg\{1+(\frac{1-\theta}{\theta})(\frac{1+e^{\alpha+\beta}}{1+e^\alpha})(\frac{\nu+e^{-\beta}}{\nu+1})\bigg\}^{-1} +  \bigg\{1+(\frac{\theta}{1-\theta})(\frac{1+e^\alpha}{1+e^{\alpha+\beta}})(\frac{\nu+e^\beta}{\nu+1})\bigg\}^{-1}\right]^{-1}}{1+\frac{1}{\nu}}\nonumber\\
  =& \bigg[ \bigg\{1+\bigg(\frac{1-\theta}{\theta}\bigg)\bigg(\frac{1+e^{\alpha+\beta}}{1+e^\alpha}\bigg)\bigg(\frac{\nu+e^{-\beta}}{\nu+1}\bigg)\bigg\}^{-1} \nonumber\\
  &+  \bigg\{1+\left(\frac{\theta}{1-\theta}\right)\bigg(\frac{1+e^\alpha}{1+e^{\alpha+\beta}}\bigg)\bigg(\frac{\nu+e^\beta}{\nu+1}\bigg)\bigg\}^{-1}\bigg]^{-1}\label{arp_AM}\\
  =& 1 + \frac{\nu\theta(1-\theta)}{(1+\nu)}\frac{(1-e^\beta)^2}{ \big\{(1-\theta)\phi + e^\beta\theta\phi^{-1}\big\}^2+\nu e^\beta \big\{ (1-\theta)\phi + \theta\phi^{-1} \big\}^2},\nonumber\\
  =& \lambda, \label{arp_AM_res}
\end{align}
where $\phi = \sqrt{\frac{1+e^{\alpha+\beta}}{1+e^{\alpha}}}$ and $\lambda \geq 1$, and $\lambda = 1$ if and only if $\beta = 0$.

Denote $\rho=e^\alpha$. In the rare outcome case ($f\to0$ or equivalently $\rho\to 0$), 
applying Taylor's expansion to (\ref{arp_AM}), we have
\begin{align*}
  \lim_{\gamma\rightarrow 0} \frac{\sigma_A^2}{\sigma_M^2} &= 
  \left[ \bigg\{1+\left(\frac{1-\theta}{\theta}\right)\left(\frac{\nu+e^{-\beta}}{\nu+1}\right)\bigg\}^{-1} +  \bigg\{1+\left(\frac{\theta}{1-\theta}\right)\left(\frac{\nu+e^\beta}{\nu+1}\right)\bigg\}^{-1}\right]^{-1}+O(\rho)\\
  &= 1+\frac{\nu\theta(1-\theta)}{(1+\nu)}\frac{(1-e^\beta)^2}{ \bigg\{(1-\theta+e^\beta\theta)^2+\nu e^\beta\bigg\}} + O(\rho)\\
  &= \lambda_0 + O(\rho),
\end{align*}
where $\lambda_0$ is defined in (12) in the main text.
Obviously, $\lambda_0\geq 1$ and $\lambda_0=1$ if and only if $\beta=0$.

Finally, we derive the asymptotic distribution of $\hat\gamma_{AC}$. The logarithm of the likelihood function (4) in the main text can be written as
\begin{align*}
l_{AC}=\sum_{i=1}^n\big[&(\alpha+\beta x_i+\gamma g_i)d_i-\log(1+\exp(\alpha+\beta x_i+\gamma g_i))\\
&+x_i\log\theta+(1-x_i)\log(1-\theta)+g_i\log(\pi)+(1-g_i)\log(1-\pi)\big],
\end{align*}
where $\theta$ is defined in (3) in the main text. It can be easily checked that the regularity conditions required for the asymptotic normality of $\hat\gamma_{AC}$ hold true \citep[see Chapter 5,][]{van2000asymptotic}. The Fisher information matrix is
$$
{I}_{AC}(\textbf{u})=-E\frac{\partial^2 l_{AC}}{\partial \textbf{u}\partial \textbf{u}^T},
$$
where $\textbf{u}=(\alpha, \beta, \gamma, \pi)^T$.
It is easy to derive that
$$
I_{AC}(\textbf{u}) =
\begin{bmatrix}
a & b & c & 0\\
b & b & d & 0\\
c & d & c & 0\\
0 & 0 & 0 & t
\end{bmatrix} + g \frac{\partial \theta}{\partial \textbf{u}}  \frac{\partial \theta}{\partial \textbf{u}^T} +h\frac{\partial^2 \theta}{\partial \textbf{u}\partial \textbf{u}^T},
$$
where
$$
a = E\left(n_{+11}\frac{e^{\alpha+\beta+\gamma}}{(1+e^{\alpha+\beta+\gamma})^2}+n_{+10}\frac{e^{\alpha+\beta}}{(1+e^{\alpha+\beta})^2}+n_{+01}
\frac{e^{\alpha+\gamma}}{(1+e^{\alpha+\gamma})^2}+n_{+00}\frac{e^{\alpha}}{(1+e^{\alpha})^2}\right),
$$
$$
b = E\left(n_{+11}\frac{e^{\alpha+\beta+\gamma}}{(1+e^{\alpha+\beta+\gamma)^2}}+n_{+10}\frac{e^{\alpha+\beta}}{(1+e^{\alpha+\beta})^2}\right),
$$
$$
c = E\left(n_{+11}\frac{e^{\alpha+\beta+\gamma}}{(1+e^{\alpha+\beta+\gamma)^2}}+n_{+01}\frac{e^{\alpha+\gamma}}{(1+e^{\alpha+\gamma})^2} \right),
$$
$$
d = E\left(n_{11+}\frac{e^{\alpha+\beta+\gamma}}{(1+e^{\alpha+\beta+\gamma)^2}}\right),
$$
$$t = E\left(\frac{n_{+11}+n_{+01}}{\pi^2}+\frac{n_{+10}+n_{+00}}{(1-\pi)^2}\right),
$$
$$
g = E\left(\frac{n_{+11}+n_{+10}}{\theta^2}+\frac{n_{+01}+n_{+00}}{(1-\theta)^2}\right),
$$
$$
h = E\left(\frac{n_{+01}+n_{+00}}{1-\theta}-\frac{n_{+11}+n_{+10}}{\theta} \right).
$$
Since
$$E(n_{+ij}) = n_{1++}p_{1ij}+n_{0++}p_{0ij}=\frac{n}{1+\nu}\left(\nu p_{1ij}+ p_{0ij}\right),$$
$$p_{1ij} = \pr(X=i,E=j\mid D=1) =  (p_{ij}\theta^i(1-\theta)^{1-i}\pi^j(1-\pi)^{1-j})/f,$$
$$p_{0ij} = pr(X=i,E=j\mid D=0) = (q_{ij}\theta^i(1-\theta)^{1-i}\pi^j(1-\pi)^{1-j})/(1-f),$$
we have that
$$
\lim_{\gamma\rightarrow 0}\lim_{f\rightarrow 0} a = \frac{e^{\alpha+\beta }n\theta}{1+\nu}\left(\frac{e^\beta \nu}{e^\beta \theta +1-\theta}+1\right)+\frac{e^\alpha n(1-\theta)}{1+\nu}\left(\frac{\nu}{e^\beta \theta +1-\theta}\right),
$$
$$
\lim_{\gamma\rightarrow 0}\lim_{f\rightarrow 0} b = \frac{e^{\alpha+\beta }n\theta}{1+\nu}\left(\frac{e^\beta \nu}{e^\beta \theta +1-\theta}+1\right),
$$
$$
\lim_{\gamma\rightarrow 0}\lim_{f\rightarrow 0} c =\frac{e^{\alpha+\beta }n\theta\pi}{1+\nu}\left(\frac{e^\beta \nu}{e^\beta \theta +1-\theta}+1\right)+\frac{e^\alpha n(1-\theta)\pi}{1+\nu}\left(\frac{\nu}{e^\beta \theta +1-\theta}\right),
$$
$$
\lim_{\gamma\rightarrow 0}\lim_{f\rightarrow 0} d =\frac{e^{\alpha+\beta }n\theta\pi}{1+\nu}\left(\frac{e^\beta \nu}{e^\beta \theta +1-\theta}+1\right),
$$
$$
\lim_{\gamma\rightarrow 0}\lim_{f\rightarrow 0} t =\frac{n}{\pi(1-\pi)},
$$
$$
\lim_{\gamma\rightarrow 0}\lim_{f\rightarrow 0} g =\frac{n}{\theta(1+\nu)}\left(\frac{e^\beta \nu}{e^\beta\theta+1-\theta}+1\right)+\frac{n}{(1-\theta)(1+\nu)}\left(\frac{\nu}{e^\beta\theta+1-\theta}+1\right),
$$
and
$$
\lim_{\gamma\rightarrow 0}\lim_{f\rightarrow 0} h = \frac{n\nu(1-e^\beta)}{(1+\nu)(e^\beta\theta+1-\theta)}.
$$
The standard likelihood theory gives that
\begin{equation}
n^{1/2}(\hat{\gamma}_{AC}-\gamma)\rightarrow N(0,\sigma_{AC}^2)\text{ in distribution},
\end{equation}
where
\begin{equation}\label{sigmaAC}
\sigma_{AC}^2=n (I_{AC})^{-1}_{33}.
\end{equation}
After tedious symbolic algebra using the software Mathematica, we have that
\begin{equation}
\lim_{\gamma\rightarrow 0}\lim_{f\rightarrow 0}n
(I_{AC})^{-1}_{33}=\sigma^2_0,
\end{equation}
where $\sigma_0^2=(2+\nu+1/\nu)/\{\pi(1-\pi)\}$. That is,
\begin{equation}\label{sigmaAC0}{}
\sigma_{AC}^2\to \sigma_0^2 \mbox{ as $f\to0$ and $\gamma\to0$}.
\end{equation}

\section{Restating Pitman's asymptotic relative efficiency}\label{sec:define_pitman}
We restate Pitman's asymptotic relative efficiency \citep{pitman1979, serfling2009} below to facilitate our discussion in the main context.
\begin{definition}
    Consider the problem of testing null hypothesis $H_0:\gamma=0$ against the alternative hypothesis $\gamma\not=0$. For a sequence of test statistics indexed by sample size $n$, $T=\{T_n\}$, suppose that (i) there exist non-random variates $\mu_n(\gamma)$ and $\sigma_n(\gamma)$ such that $n^{1/2}(T_n-\mu_n(\gamma))/\sigma_n(\gamma)$ converges in distribution to the standard normal distribution as $n \rightarrow \infty$ under the contiguous alternative hypothesis $H_1:\gamma=cn^{-1/2}$, (ii) $\mu_n(\gamma)$ has a continuous derivative $\mu_n'(\gamma)$ in a neighbourhood of 0, and (iii) $\sigma_n(\gamma)$ is continuous at $0$. Then $n^{1/2}\sigma_n(0)/\mu_n'(0)$ converges to some constant as $n\rightarrow \infty$. Let $\kappa_A$ and $\kappa_B$ denote such constants corresponding to test statistic sequences $T_A$ and $T_B$, respectively. Pitman's asymptotic relative efficiency of $T_A$ to $T_B$ is defined as $e_P(T_A, T_B) = \left({\kappa_B}/{\kappa_A}\right)^2$.
\end{definition}

\section{Proof of Theorem~1 and Corollary~3}\label{sec:proof_theorem1_coro3}

Adopting the previous notations
\begin{equation}
\rho=e^\alpha,~ b_1=e^\beta(1-\theta)+\theta, \mbox{ and } b_2={e^\beta}/(e^\beta\theta-\theta+1),
\end{equation}
then for $\delta$ defined in (\ref{delta_sup}) we have that
\begin{align}
\lim_{\gamma\rightarrow 0}\frac{d(\gamma+\delta)}{d\gamma}
&= \lim_{\gamma\rightarrow0} \bigg[1+\frac{d}{d\gamma}\log\bigg\{ 1 + \frac{e^\alpha(b_1 - b_2)(1-e^\gamma)}{(1+e^{\alpha+\gamma} b_1)(1+e^\alpha b_2)}\bigg\}\bigg]\nonumber\\
&=1-\frac{(b_1-b_2)\rho}{(1+b_1\rho)(1+b_2\rho)}\nonumber\\
&=\frac{b_1b_2\rho^2+2b_2\rho+1}{b_1b_2\rho^2+(b_1+b_2)\rho+1}.\label{partial_bias}
\end{align}

By (\ref{partial_bias}) and Lemma~2, Pitman's asymptotic relative efficiency of \textsc{Mar} to \textsc{Adj} is equal to
\begin{align*}
  e_P(\hat{\gamma}_M,\hat{\gamma}_A)&= \bigg\{\lim_{\gamma \rightarrow 0}\left(\frac{d(\gamma+\delta)/{d\gamma}}{d\gamma/d\gamma}\right)\bigg\}^2 \bigg\{\lim_{\gamma\rightarrow 0}\frac{\mathrm{var}(\hat{\gamma}_A)}{\mathrm{var}(\hat{\gamma}_M)}\bigg\}\\
  &=  \bigg\{\frac{b_1b_2\rho^2+2b_2\rho+1}{b_1b_2\rho^2+(b_1+b_2)\rho+1}\bigg\}^2 \lambda.
\end{align*}

In the rare outcome situation ($f\to0$ or equivalently $\rho\to0$), applying Taylor’s expansion to (\ref{partial_bias}), we have
\begin{equation}
  \lim_{\gamma\rightarrow 0}\frac{d(\gamma+\delta)}{d\gamma}=1-(b_1-b_2)\rho+(b_1^2-b_2^2)\rho^2+O(\rho^3).
\end{equation}
Consequently,
\begin{equation}\label{tmp}
\left(\lim_{\gamma\rightarrow 0} \frac{d(\gamma+\delta)}{d\gamma}\right)^2 = 1-2(b_1-b_2)\rho +  \{2(b_1^2-b_2^2)+(b_1-b_2)^2\}\rho^2+O(\rho^3).
\end{equation}
By (\ref{tmp}) and Corollary~2, when the outcome is rare, as indicated by a small $\rho,$ Pitman's asymptotic relative efficiency of \textsc{Mar} to \textsc{Adj} is equal to
\begin{align*}
e_P(\hat{\gamma}_M,\hat{\gamma}_A)&= \bigg\{\lim_{\gamma \rightarrow 0}\left(\frac{d(\gamma+\delta)/{d\gamma}}{d\gamma/d\gamma}\right)\bigg\}^2 \bigg\{\lim_{\gamma\rightarrow 0}\frac{\mathrm{var}(\hat{\gamma}_A)}{\mathrm{var}(\hat{\gamma}_M)}\bigg\}\\
&=\left(1-2(b_1-b_2)\rho+O(\rho^2)\right)\left(\lambda_0 +O(\rho)\right)\\
&=\lambda_0 + O(\rho),
\end{align*}
where $\lambda_0$ is defined in (12) of the main text.

\section{Proof of Theorem 2}\label{sec:proof_theorem2}

We adopt the notations in the proof of Lemma 2:
$$
\sigma_{AC}^2=\lim_{n\rightarrow \infty}
\mathrm{var}(n^{1/2}\hat{\gamma}_{AC}),\quad\sigma_M^2=\lim_{n\rightarrow \infty}\mathrm{var}(n^{1/2}\hat{\gamma}_M).
$$
The second-order Taylor expansion of $\sigma_{AC}^2$ with respect to $f$ is
$$
\sigma_{AC}^2=\sigma_{AC}^2|_{f=0}+\frac{\partial }{\partial f}\sigma_{AC}^2\Big|_{f=0}\times f+\frac{1}{2}\frac{\partial^2}{\partial f^2} \sigma_{AC}^2\Big|_{f=0}\times f^2+O(f^3),
$$
so that
\begin{align}
&e_P(\hat{\gamma}_M,\hat{\gamma}_{AC})\nonumber\\
=&\left(\lim_{\gamma\rightarrow 0} \frac{d(\gamma+\delta)/d\gamma}{d\gamma/{d\gamma}}\right) ^2 \times \lim_{\gamma\rightarrow 0}\frac{\mathrm{var}(\hat{\gamma}_{AC})}{\mathrm{var}(\hat{\gamma}_M)} \nonumber\\
=& \bigg\{1-2(b_1-b_2)\rho+\left\{2(b_1^2-b_2^2)+(b_1-b_2)^2\right\}\rho^2+O(\rho^3)\bigg\}\nonumber\\
&\times\lim_{\gamma\to0} \left(\frac{\sigma_{AC}^2|_{f=0}}{\sigma_M^2}+\frac{\frac{\partial}{\partial f} \sigma_{AC}^2|_{f=0}}{\sigma_M^2}\times f+\frac{1}{2}\frac{\frac{\partial^2}{\partial f^2} \sigma_{AC}^2|_{f=0}}{\sigma_M^2}\times f^2+O(f^3)\right)\label{eq0}
\end{align}
by (\ref{tmp}).
Symbolic algebra with the software Mathematica gives
\begin{equation}\label{sigmaAC1}
\frac{\partial}{\partial f} \sigma_{AC}^2\Big|_{f=0,\gamma=0}=\frac{2(2+\nu+1/\nu) (\theta -1) \theta  \left(e^{\beta }-1\right)^2}{\left(\theta  \left(e^{\beta }-1\right)+1\right)^2 (\pi -1) \pi}
\end{equation}
and
\begin{align}
&\frac{\partial^2}{\partial f^2} \sigma_{AC}^2\Big|_{f=0,\gamma=0}\nonumber\\
=&-\big[e^{\beta } \nu^2+\theta ^2 (e^{\beta }-1 )^2 (2 \nu+1)-\theta   (e^{\beta }-1\ )  \{ (e^{\beta }-3 ) \nu-2\}+5 e^{\beta } \nu+\nu+1 \big]\nonumber\\
&\times\frac{2 (\theta -1) \theta  \left(e^{\beta }-1\right)^2 (\nu+1)^2 }{\left(\theta  \left(e^{\beta }-1\right)+1\right)^4 (\pi -1) \pi \nu^2}.\label{sigmaAC2}
\end{align}
If $\gamma=0$, then the outcome prevalence can be expressed as
$$
f|_{\gamma=0}=\frac{e^{\alpha+\beta}}{1+e^{\alpha+\beta}}\theta+\frac{e^\alpha}{1+e^\alpha}(1-\theta) ,\label{f1}
$$
so that
\begin{equation}\label{f}
f|_{\gamma=0}=(e^\beta\theta-\theta+1)\rho-(e^{2\beta}\theta-\theta+1)\rho^2 + O(\rho^3).
\end{equation}
It follows from (\ref{sigmaM0}) and (\ref{sigmaAC0}) that
\begin{equation}\label{eq1}
\sigma_M^2|_{\gamma= 0} = \sigma_0^2\mbox{ and }
\sigma_{AC}^2|_{f=0,\gamma=0}={\sigma_0^2}.
\end{equation}
Furthermore, from (\ref{sigmaAC1})-(\ref{f}), we have
\begin{align}\label{eq2}
&\frac{\frac{\partial}{\partial f} \sigma_{AC}^2|_{f=0,\gamma=0}}{\sigma_0^2}\times f\nonumber\\
&=\frac{2(1-\theta) \theta  \left(e^{\beta }-1\right)^2}{\theta  \left(e^{\beta }-1\right)+1}\rho+\frac{2(1-\theta)\theta(e^\beta-1)^2(\theta-1-e^{2\beta}\theta)}{(\theta(e^\beta-1)+1)^2}\rho^2+O(\rho^3)\nonumber\\
&=2(b_1-b_2)\rho - \frac{2(1-\theta)\theta(e^\beta-1)^2(1+(e^{2\beta}-1)\theta)}{(\theta(e^\beta-1)+1)^2}\rho^2 +O(\rho^3)
\end{align}
and
\begin{align}
&\frac{1}{2}\frac{\frac{\partial^2}{\partial f^2} \sigma_{AC}^2|_{f=0}}{\sigma_0^2}\times f^2 \nonumber\\
=&\big[(e^{\beta } \nu^2+\theta ^2 (e^{\beta }-1)^2 (2 \nu+1)-\theta  (e^{\beta }-1) \{(e^{\beta }-3) \nu-2\}+5 e^{\beta } \nu+\nu+1\big]\nonumber\\
&\times\frac{(\theta -1) \theta  (e^{\beta }-1)^2 }{\{\theta (e^{\beta }-1)+1\}^2 \nu}\rho^2+O(\rho^3). \label{eq3}
\end{align}
By equations (\ref{eq0}) and (\ref{eq1})-(\ref{eq3}), we have
\begin{align*}
e_P(\hat{\gamma}_M,\hat{\gamma}_{AC})=1+\tau\rho^2+O(\rho^3),
\end{align*}
where
\begin{equation}\label{tau}
    \tau = -\frac{(1-\theta) \theta  (e^{\beta }-1)^2  \{(1+1/\nu) [(\theta (e^{\beta }-1)+1)^2+e^{\beta } \nu]+2(1+(e^{2\beta}-1)\theta)\}}{ \{\theta  \left(e^{\beta }-1\right)+1\}^2}.
\end{equation}
Obviously, $\tau\leq0$ and the equality holds if and only if $\beta=0$.

\section{Robustness of \textsc{AdjCon} with respect to  prevalence setting}\label{sec:proof_theorem3}
\subsection{Notations and preliminary results}
Let $\boldsymbol s = (\beta,\gamma,\theta,\pi)^\top$ denote unknown model parameters. Denote by $f_0$ the true outcome prevalence, which is incorrectly specified as $f_1$ in \textsc{AdjCon}. In this section, we show that \textsc{AdjCon} is robust with respect to the misspecification of the outcome prevalence. Let $B$ be the domain of $(\beta,\gamma,\theta,\pi)$: $B=\{(\beta,\gamma,\theta,\pi)\mid \beta$ and $\gamma$ are bounded away from infinity, $\theta$ and $\pi$ are bounded away from zero and one$\}$.

Assume that $f_0,f_1\in (0, 1-\epsilon]$ for some give $\epsilon>0$, which is easily hold in practice. Assume that $\bfs_f^*=(\beta^*,\gamma^*,\theta^*,\pi^*)\in B$ for any $f\in(0,1-\epsilon]$. The log-likelihood is
\begin{align*}
    l_f(\bfs) &=(\alpha+\beta X+\gamma E)D-\log(1+\exp(\alpha+\beta X+\gamma E))+\\
    &\quad X\log\theta+(1-X)\log(1-\theta)+E\log(\pi)+(1-E)\log(1-\pi),
\end{align*}
subject to the prevalence constraint 
\begin{equation}\label{eq:f_constraint}
  f = \sum_{i=0}^1\sum_{j=0}^1 p(D=1\mid X=i,E=j)p(X=i)p(E=j).
\end{equation}

We will show that both $\bfs_{f_1}^*$ and the corresponding asymptotic covariance matrix $\Sigma_{f_1}(\bfs_{f_1}^*)$ are Lipschitz continuous with respect to $f_1$, that is $$\Vert \bfs_{f_1}^* - \bfs_{f_0}^*\Vert = C_1\vert f_1 - f_0\vert$$
and
\begin{equation}
  \Vert\Sigma_{f_1}(\bfs_{f_1}^*) - \Sigma_{f_0}(\bfs_{f_0}^*)\Vert = C_2\vert f_1-f_0\vert,
\end{equation}
where 
\begin{equation}\label{eq:bfs01}
  \bfs_{f_0}^* = \arg\max_\bfs E_{f_0}\{l_{f_0}(\bfs)\},\quad \bfs_{f_1}^* = \arg\max_\bfs E_{f_0}\{l_{f_1}(\bfs)\},
\end{equation}
and $C_1,C_2$ are independent of $f_0$ and $f_1.$

We define the following quantities that will be used in the proof of Lemma \ref{lemma1}. Let 
$$M_1(\bfs) = e^{\beta+\gamma}\theta\pi+e^{\beta}\theta(1-\pi)+e^{\gamma}(1-\theta)\pi+(1-\theta)(1-\pi),$$ then
$$
0<m_1=\min_{\bfs\in B}\{e^{\beta+\gamma},e^{\beta},e^{\gamma},1\}\leq M_1(\bfs)\leq \max_{\bfs\in B}\{e^{\beta+\gamma},e^{\beta},e^{\gamma},1\} = M_1.
$$
Let 
$$M_2(\bfs) = e^{-\beta-\gamma}\theta\pi+e^{-\beta}\theta(1-\pi)+e^{-\gamma}(1-\theta)\pi+(1-\theta)(1-\pi),$$ then
$$
0<m_2=\min_{\bfs\in B}\{e^{-\beta-\gamma},e^{-\beta},e^{-\gamma},1\}\leq M_2(\bfs)\leq \max_{\bfs\in B}\{e^{-\beta-\gamma},e^{-\beta},e^{-\gamma},1\} = M_2.
$$

The following lemma presents a decomposition of the intercept parameter $\alpha$ under the prevalence constraint.
\begin{lemma}\label{lemma1_sup}
  Assume $f\in(0,1-\epsilon]$ for some $\epsilon>0$,  $\beta$ and $\gamma$ are bounded away from infinity and $\theta$ and $\pi$ are bounded away from zero and one. Denote  $\bfs=(\beta,\gamma,\theta,\pi)$. The intercept $\alpha$, as a function of $f$ and $\bfs$ due to constraint
  \begin{equation}\label{constrain_f}
  f = \sum_{i=0}^1\sum_{j=0}^1\pr(D=1\mid X=i,E=j)\pr(X=i)\pr(E=j),
  \end{equation}
  can be decomposed into two parts: $$\alpha(f,\bfs) = \alpha_1(f) + \alpha_2(f,\bfs),$$ where $\alpha_2(f,\bfs)$ is Lipschitz continuous with respect to $f$.
\end{lemma}

\begin{proof}
Denote $\rho(f,\bfs) = \exp(\alpha(f,\bfs))$ and rewrite the constraint \eqref{constrain_f}:
  \begin{align}
      f=& F(\alpha,\bfs) = \sum_{i=0}^1\sum_{j=0}^1 p(D=1|X=i,E=j)p(X=i)p(E=j)\notag\\
      =&\frac{\exp(\alpha+\beta+\gamma)\theta\pi}{1+\exp(\alpha+\beta+\gamma)} + \frac{\exp(\alpha+\beta)\theta(1-\pi)}{1+\exp(\alpha+\beta)}+\frac{\exp(\alpha+\gamma)(1-\theta)\pi}{1+\exp(\alpha+\gamma)}\notag\\
      &+\frac{\exp(\alpha)(1-\theta)(1-\pi)}{1+\exp(\alpha)}\notag\\
      =& \frac{\rho}{1+\rho}\bigg(\frac{(1+\rho)e^{\beta+\gamma}\theta\pi}{1+\rho e^{\beta+\gamma}}+\frac{(1+\rho)e^{\beta}\theta(1-\pi)}{1+\rho e^{\beta}}+\frac{(1+\rho)e^{\gamma}(1-\theta)\pi}{1+\rho e^{\gamma}}+(1-\theta)(1-\pi)\bigg)\notag\\
      =& \frac{\rho}{1+\rho} C'(\rho,\bfs),\label{eq:f_constrain}
  \end{align} 
where as $\rho$ ranges from 0 to $\infty$, $C'(\rho,\bfs)$ ranges from $M_1(\bfs)$ to $1$. Similarly, 
\begin{align}
1-f = &\frac{1}{1+\rho}\bigg[\frac{(1+\rho)\theta\pi}{1+\rho e^{\beta+\gamma}}+\frac{(1+\rho)\theta(1-\pi)}{1+\rho e^\beta}+\frac{(1+\rho)(1-\theta)\pi}{1+\rho e^\gamma}+(1-\theta)(1-\pi)\bigg] \notag\\
=& \frac{1}{1+\rho} C''(\rho,\bfs),\label{eq:1-f_constrain}
\end{align}
where as $\rho$ ranges from 0 to $\infty$, $C''(\rho,\bfs)$ ranges from $1$ to $M_2(\bfs)$.
Combining \eqref{eq:f_constrain} and \eqref{eq:1-f_constrain}, we have 
\begin{align*}
  \alpha(f,\bfs) &= \{\log f - \log(1-f)\} + \{\log C''(\rho(f,\bfs),\bfs) - \log C'(\rho(f,\bfs),\bfs)\}\\
  &= \alpha_1(f) + \alpha_2(f,\bfs).
\end{align*}
In what follows, we show that
\begin{equation}\label{lip1}
\alpha_2(f,\bfs) = \log C''(\rho(f,\bfs),\bfs) - \log C'(\rho(f,\bfs),\bfs)
\end{equation}
is Lipschitz continuous with respect to $f$.
First,
    \begin{align}\label{lip2}
        &\bigg\vert\frac{\partial\log C''(\rho,\bfs)}{\partial f}\bigg\vert= \bigg\vert\frac{1}{C''(\rho,\bfs)}\frac{\partial C''(\rho,\bfs)}{\partial \rho}\frac{\partial \rho}{\partial f}\bigg\vert\notag\\
        =& \bigg\vert\frac{1}{C''(\rho,\bfs)}\frac{\frac{\theta\pi(1-e^{\beta+\gamma})}{(1+\rho e^{\beta+\gamma})^2}+\frac{\theta(1-\pi)(1-e^{\beta})}{(1+\rho e^{\beta})^2}+\frac{(1-\theta)\pi(1-e^{\gamma})}{(1+\rho e^{\gamma})^2}}{\frac{e^{\beta+\gamma}\theta\pi}{(1+\rho e^{\beta+\gamma})^2}+\frac{e^{\beta}\theta(1-\pi)}{(1+\rho e^{\beta})^2}+\frac{e^{\gamma}(1-\theta)\pi}{(1+\rho e^{\gamma})^2}+\frac{(1-\theta)(1-\pi)}{(1+\rho)^2}}\bigg\vert\notag\\
        =& \frac{1}{C''(\rho,\bfs)}\frac{1}{C'''(\rho,\bfs)}\bigg\vert\frac{(1+\rho)^2\theta\pi}{(1+\rho e^{\beta+\gamma})^2}+\frac{(1+\rho)^2\theta(1-\pi)}{(1+\rho e^{\beta})^2}+\frac{(1+\rho)^2(1-\theta)\pi}{(1+\rho e^{\gamma})^2}\notag\\
        &+(1-\theta)(1-\pi)-C'''(\rho,\bfs)\bigg\vert\notag\\
        \leq& \frac{1}{m_2}\frac{1}{\min\{m_1,m_2\}}(M_2^2+M_1+M_2),
    \end{align}
where 
\begin{align*}
  C'''(\rho,\bfs)=\bigg\{&\frac{(1+\rho)^2e^{\beta+\gamma}\theta\pi}{(1+\rho e^{\beta+\gamma})^2}+\frac{(1+\rho)^2e^{\beta}\theta(1-\pi)}{(1+\rho e^{\beta})^2}+\frac{(1+\rho)^2e^{\gamma}(1-\theta)\pi}{(1+\rho e^{\gamma})^2}\\
  &+(1-\theta)(1-\pi)\bigg\}.
\end{align*}
As $\rho$ ranges from 0 to $\infty$, $C'''(\rho,\bfs)$ ranges from $M_1(\bfs)$ to $M_2(\bfs).$
Second,
        \begin{align}\label{lip3}
            &\bigg\vert\frac{\partial\log C'(\rho,\bfs)}{\partial f}\bigg\vert= \bigg\vert\frac{1}{C'(\rho,\bfs)}\frac{\partial C'(\rho,\bfs)}{\partial \rho}\frac{\partial \rho}{\partial f}\bigg\vert\notag\\
            =& \bigg\vert\frac{1}{C'(\rho,\bfs)}\frac{\frac{e^{\beta+\gamma}\theta\pi(1-e^{\beta+\gamma})}{(1+\rho e^{\beta+\gamma})^2}+\frac{e^\beta\theta(1-\pi)(1-e^{\beta})}{(1+\rho e^{\beta})^2}+\frac{e^\gamma(1-\theta)\pi(1-e^{\gamma})}{(1+\rho e^{\gamma})^2}}{\frac{e^{\beta+\gamma}\theta\pi}{(1+\rho e^{\beta+\gamma})^2}+\frac{e^{\beta}\theta(1-\pi)}{(1+\rho e^{\beta})^2}+\frac{e^{\gamma}(1-\theta)\pi}{(1+\rho e^{\gamma})^2}+\frac{(1-\theta)(1-\pi)}{(1+\rho)^2}}\bigg\vert\notag\\
            =& \frac{1}{C'(\rho,\bfs)}\frac{1}{C'''(\rho,\bfs)}\bigg\vert C'''(\rho,\bfs)
            -\bigg[\frac{e^{2\beta+2\gamma}(1+\rho)^2\theta\pi}{(1+\rho e^{\beta+\gamma})^2}+\frac{e^{2\beta}(1+\rho)^2\theta(1-\pi)}{(1+\rho e^{\beta})^2}\notag\\
            &\quad\quad\quad\quad\quad\quad\quad\quad\quad\quad\quad\quad\quad+\frac{e^{2\gamma}(1+\rho)^2(1-\theta)\pi}{(1+\rho e^{\gamma})^2}+(1-\theta)(1-\pi)\bigg]\bigg\vert\notag\\
            \leq& \frac{1}{m_1}\frac{1}{\min\{m_1,m_2\}}(M_1^2+M_1+M_2).
        \end{align}
        It follows from \eqref{lip1}-\eqref{lip3} that $\alpha_2(f,\bfs)$ is Lipschitz continuous with respect to $f\in (0,1-\epsilon]$ for some $\epsilon>0$ and any $\bfs\in B$. Denote the Lipschitz constant by 
        \begin{equation}\label{L_C}
        L_C=\frac1{\min\{m_1,m_2\}}\bigg(\frac{1}{m_2}(M_2^2+M_1+M_2)+\frac{1}{m_1}(M_1^2+M_1+M_2)\bigg).
        \end{equation}
\end{proof}
\subsection{Proof of Theorem 3: Part I (Lipschitz continuity of $\bfs_{f_1}^*$)}\label{sec:7.2}

Before proving $\Vert \bfs_{f_1}^* - \bfs_{f_0}^*\Vert \leq C_1\vert f_1 - f_0\vert$, we first prove
\begin{equation}
     E_{f_0}[l_{f_0}(\bfs_{f_0}^*) - l_{f_0}(\bfs_{f_1}^*)] \leq C|f_1 - f_0|.
\end{equation}
Since $ E_{f_0}[l_{ f_1}(\bfs_{f_0}^*)] \leq  E_{f_0}[l_{ f_1}(\bfs_{ f_1}^*)]$, we have
\begin{equation}
     E_{f_0}[l_{f_0}(\bfs_{f_0}^*) + (l_{f_1}(\bfs_{f_0}^*)-l_{f_0}(\bfs_{f_0}^*))] \leq  E_{f_0}[l_{f_0}(\bfs_{ f_1}^*)+(l_{f_1}(\bfs_{f_1}^*) - l_{f_0}(\bfs_{f_1}^*))].
\end{equation}
Thus,
$$
    0\leq E_{f_0}[l_{f_0}(\bfs_{f_0}^*) - l_{f_0}(\bfs_{f_1}^*)]\leq  E_{f_0}[(l_{f_1}(\bfs_{f_1}^*)-l_{f_0}(\bfs_{f_1}^*))] -  E_{f_0}[(l_{f_1}(\bfs_{f_0}^*)-l_{f_0}(\bfs_{f_0}^*))], 
$$
where the first equality holds according to the definition of \eqref{eq:bfs01}. We need only to prove that the right-hand side of the above inequality is Lipschitz continuous with respect to $f$.

For any $\bfs\in B$, we have
    \begin{align*}
         E_{f_0}&[l_{f_1}(\bfs) - l_{f_0}(\bfs)] = E_{f_0}\bigg[\frac{\partial l_f(\bfs)}{\partial f}\bigg\vert_{f=f^*}\bigg](f_1-f_0)\\
        &=  E_{f_0}\bigg[(\alpha_1 - \alpha_0)D - \log\frac{1+\exp(\alpha_1+\beta X+\gamma E)}{1 + \exp( \alpha_0+\beta X+\gamma E)}\bigg] \\
        &= \nu/(1+\nu)(\alpha_1 - \alpha_0) -  E_{f_0}\bigg[\log\frac{1+\exp(\alpha_1+\beta X+\gamma E)}{1 + \exp( \alpha_0+\beta X+\gamma E)}\bigg],
    \end{align*}
so that
\begin{align}
&E_{f_0}[(l_{f_1}(\bfs_{f_1}^*)-l_{f_0}(\bfs_{f_1}^*))] -  E_{f_0}[(l_{f_1}(\bfs_{f_0}^*)-l_{f_0}(\bfs_{f_0}^*))]\notag\\
=
    & E_{f_0}\bigg[\nu/(1+\nu)(\alpha(f_1,\bfs_{f_1}^*) - \alpha(f_0,\bfs_{f_1}^*)) - \log\frac{1+\exp(\alpha(f_1,\bfs_{f_1}^*)+\beta X+\gamma E)}{1 + \exp( \alpha(f_0,\bfs_{f_1}^*)+\beta X+\gamma E)}\bigg]\notag\\
    & -E_{f_0}\bigg[\nu/(1+\nu)(\alpha(f_1,\bfs_{f_0}^*) - \alpha(f_0,\bfs_{f_0}^*)) - \log\frac{1+\exp(\alpha(f_1,\bfs_{f_0}^*)+\beta X+\gamma E)}{1 + \exp( \alpha(f_0,\bfs_{f_0}^*)+\beta X+\gamma E)}\bigg]\notag\\
    =&\frac{\nu}{1+\nu}\big([\alpha(f_1,\bfs_{f_1}^*)-\alpha(f_0,\bfs_{f_1}^*)]-[\alpha(f_1,\bfs_{f_0}^*)-\alpha(f_0,\bfs_{f_0}^*)]\big)\notag\\ 
    &-\bigg\{ E_{f_0}\bigg[\frac{\rho_1^*\exp(\beta_1^*X+\gamma_1^*E)}{1+\rho_1^*\exp(\beta_1^*X+\gamma_1^*E)}\bigg]\frac{\partial\alpha(f,\bfs_{f_1}^*)}{\partial f}\bigg\vert_{f=f_1^*} \notag\\
    &-E_{f_0}\bigg[\frac{\rho_0^*\exp(\beta_0^*X+\gamma_0^*E)}{1+\rho_0^*\exp(\beta_0^*X+\gamma_0^*E)}\bigg]\frac{\partial\alpha(f,\bfs_{f_0}^*)}{\partial f}\bigg\vert_{f=f_0^*}\bigg\}(f_1-f_0).\label{eq:diff_eq}
\end{align}

When $0<f_1,f_0<1-\epsilon,\bfs_{f_1}^*,\bfs_{f_0}^*\in B$, in what follows, we show both of the two terms in the righthand side of \eqref{eq:diff_eq} are Lipschitz continuous with respect to $f$. First, according to Lemma~\ref{lemma1_sup},
\begin{equation}
    \alpha(f,\bfs) = \alpha_1(f) + \alpha_2(f,\bfs).
\end{equation}
For the first term in \eqref{eq:diff_eq},
\begin{align*}
    &[\alpha(f_1,\bfs_{f_1}^*)-\alpha(f_0,\bfs_{f_1}^*)]-[\alpha(f_1,\bfs_{f_0}^*)-\alpha(f_0,\bfs_{f_0}^*)]\\
    =&[\alpha(f_1,\bfs_{f_1}^*)-\alpha(f_1,\bfs_{f_0}^*)]-[\alpha(f_0,\bfs_{f_1}^*)-\alpha(f_0,\bfs_{f_0}^*)]\\
    =&[\alpha_2(f_1,\bfs_{f_1}^*) - \alpha_2(f_0,\bfs_{f_1}^*)]-[\alpha_2(f_1,\bfs_{f_0}^*)-\alpha_2(f_0,\bfs_{f_0}^*)]\\
    \leq& 2L_C\vert f_1 - f_0\vert
\end{align*}
holds because $\alpha_2(f,\bfs)$ is Lipschitz continuous with respect to $f$. Next, we show the second term is also Lipschitz continuous with respect to $f$. Note that
\begin{equation}\label{eq:diff_alpha}
\begin{aligned}
    \frac{\partial\alpha(f,\bfs)}{\partial f} &= \bigg\{\frac{\rho}{(1+\rho)^2}\bigg[\frac{(1+\rho)^2e^{\beta+\gamma}\theta\pi}{(1+\rho e^{\beta+\gamma})^2}+\frac{(1+\rho)^2e^{\beta}\theta(1-\pi)}{(1+\rho e^{\beta})^2}+\frac{(1+\rho)^2e^{\gamma}(1-\theta)\pi}{(1+\rho e^{\gamma})^2}\\
    &\quad\quad\quad\quad\quad\quad+(1-\theta)(1-\pi)\bigg]\bigg\}^{-1}\\
    =& \frac{(1+\rho)^2}{\rho C'''(\rho,\bfs)} = \frac{1+\rho}{\rho}\frac{1+\rho}{C'''(\rho,\bfs)}= \frac{1+\rho}{\rho}\frac{C''(\rho,\bfs)}{(1-f)C'''(\rho,\bfs)}.
\end{aligned}
\end{equation}

According to Equation~\eqref{eq:diff_alpha}, the second term is
\begin{align*}
    &\bigg\{ E_{f_0}\bigg[\frac{(1+\rho_1^*)\exp(\beta_1^*X+\gamma_1^*E)}{1+\rho_1^*\exp(\beta_1^*X+\gamma_1^*E)}\bigg]\frac{C''(\rho_1^*,\bfs_{f_1}^*)}{(1-f_1^*)C'''(\rho_1^*,\bfs_{f_1}^*)}\\
    & \quad -E_{f_0}\bigg[\frac{(1+\rho_0^*)\exp(\beta_0^*X+\gamma_0^*E)}{1+\rho_0^*\exp(\beta_0^*X+\gamma_0^*E)}\bigg]\frac{C''(\rho_0^*,\bfs_{f_0}^*)}{(1-f_0^*)C'''(\rho_0^*,\bfs_{f_0}^*)}\bigg\}(f_1-f_0)\\
    \leq& 2L_1 L_\epsilon L_3\vert f_1-f_0\vert,
\end{align*}
where 
$$
L_1 = \max_{\bfs^* \in B}\big\{1,  E_{f_0}\exp(\beta^*X+\gamma^*E)\big\}\leq \max_{\bfs^*\in B}\{1, \exp(\beta^*+\gamma^*),\exp(\beta^*),\exp(\gamma^*)\} = M_1,
$$
$$
L_\epsilon = \max\bigg\{\frac{1}{1-f_1},\frac{1}{1-f_0}\bigg\}\leq \frac{1}{\epsilon},\quad \text{if } f_1,f_0\leq 1-\epsilon,
$$
$$
L_3 = \max_{\rho^*>0,\bfs^*\in B}\frac{C''(\rho^*,\bfs^*)}{C'''(\rho^*,\bfs^*)} = \frac{M_2}{\min\{m_1,m_2\}}.
$$
So we have 
$$
 E_{f_0} \{l_{f_0}(\bfs_{f_0}^*) - l_{f_0}(\bfs_{f_1}^*)\}\leq C\vert f_1-f_0\vert,
$$
where $C = 2\nu L_C/(1+\nu) + 2L_1L_\epsilon L_3$. 

Taylor's expansion of $l_{f_0}(\bfs_{f_1}^*)$ at $\bfs_{f_0}^*$ gives that
\begin{equation}
\begin{aligned}
    l_{f_0}(\bfs_{f_1}^*) = l_{f_0}(\bfs_{f_0}^*) + (\bfs_{f_1}^* - \bfs_{f_0}^*)^\top\frac{\partial l_{f_0}(\bfs)}{\partial\bfs}\bigg\vert_{\bfs=\bfs_{f_0}^*} + (\bfs_{f_1}^* - \bfs_{f_0}^*)^\top\frac{\partial^2 l_{f_0}(\bfs)}{\partial\bfs\partial\bfs^\top}\bigg\vert_{\bfs=\bfs^\prime}(\bfs_{f_1}^*-\bfs_{f_0}^*),
\end{aligned}
\end{equation}
where $\bfs^\prime$ lies in between $\bfs_{f_0}^*$ and $\bfs_{f_1}^*$. Consequently, 
\begin{align}
     E_{f_0} \{l_{f_0}(\bfs_{f_0}^*) - l_{f_0}(\bfs_{f_1}^*)\} &= (\bfs_{f_1}^* - \bfs_{f_0}^*)^\top \bigg\{- E_{f_0}\frac{\partial^2 l_{f_0}(\bfs)}{\partial\bfs\partial\bfs^\top}\bigg\vert_{\bfs=\bfs^\prime}\bigg\}(\bfs_{f_1}^*-\bfs_{f_0}^*)\notag\\
    &\leq C\vert f_1-f_0\vert.\label{eq:s_bound}
\end{align} 
It can be easily show that the matrix $- E_{f_0} \{\partial^2 l_{f_0}(\bfs)/(\partial\bfs\partial\bfs^\top)\vert_{\bfs=\bfs^\prime}\}$ is positive definite. 
Let the corresponding smallest eigenvalue be $\lambda_{min}>0$, then \eqref{eq:s_bound} implies 
\begin{equation}
    \Vert\bfs_{f_1}^*-\bfs_{f_0}^*\Vert \leq \frac{C}{\lambda_{min}}\vert f_1-f_0\vert.
\end{equation}
Finally, let $C_1 = C/\lambda_{\min}$ which is given in Theorem 3 of the paper. 

\subsection{Proof of Theorem 3 Part II (Lipschitz continuity of $\Sigma_{f_1}(\bfs_{f_1}^*)$)}
We prove the asymptotic covariance matrix $\Sigma_{f_1}(\bfs_{f_1}^*)$ is Lipschitz continuous with respect to $f_1$. 
Let $\hat\bfs_{f_1}$ maximize the log-likelihood function
\begin{align*}
        l_{n,f_1}=&\sum_{i=1}^n\Big[( \alpha_1+\beta x_i+\gamma g_i)d_i-\log(1+\exp( \alpha_1+\beta x_i+\gamma g_i))\\
        &\quad\quad+x_i\log\theta+(1-x_i)\log(1-\theta)+g_i\log(\pi)+(1-g_i)\log(1-\pi)\big]
\end{align*}
with the outcome prevalence being specified to be $f_1$.

According to \cite{white1982maximum}, the maximum likelihood estimator $\hat\bfs_{f_1}$ is consistent for $\bfs_{f_1}^*$ and asymptotically normal:
\begin{equation}
    \sqrt{n}(\hat{\bfs}_{ f_1} - \bfs_{f_1}^*)\rightarrow N(0, \Sigma_{ f_1}(\bfs_{f_1}^*)),
\end{equation}
where 
\begin{equation}\label{sigma_f1}
\Sigma_{ f_1}(\bfs_{f_1}^*) = A^{-1}(f_1,\bfs_{f_1}^*)B(f_1,\bfs_{f_1}^*)A^{-1}(f_1,\bfs_{f_1}^*)
\end{equation}
with
\begin{equation}
    \begin{aligned}
        A(f,\bfs) = \frac1n E_{f_0} \bigg\{\frac{\partial^2 l_{n, f}(\bfs)}{\partial \bfs\partial\bfs^\top}\bigg\}\text{ and } B(f,\bfs) = \frac1n E_{f_0} \bigg\{\frac{\partial l_{n, f}(\bfs)}{\partial\bfs}\frac{\partial l_{n, f}(\bfs)}{\partial\bfs^\top}\bigg\}.
    \end{aligned}
\end{equation}
Assume that $A(f,\bfs)$ and $B(f,\bfs)$ have good condition numbers among all $f\in(0,1-\epsilon]$ and $\bfs\in B$. Specifically, $\Vert A(f,\bfs)\Vert \leq \Lambda_A, \Vert A^{-1}(f,\bfs)\Vert \leq \lambda_A$ and $\Vert B(f,\bfs)\Vert \leq \Lambda_B, \Vert B^{-1}(f,\bfs)\Vert \leq \lambda_B$.
Similarly, we have 
\begin{equation}\label{sigma_f0}
\Sigma_{f_0}(\bfs_{f_0}^*) = A^{-1}(f_0,\bfs_{f_0}^*)B(f_0,\bfs_{f_0}^*)A^{-1}(f_0,\bfs_{f_0}^*).
\end{equation}

Note that 
\begin{equation}\label{information}
-A(f_0,\bfs_{f_0}^*) = B(f_0,\bfs_{f_0}^*).
\end{equation}
In order to show $\Vert\Sigma_{f_1}(\bfs_{f_1}^*) - \Sigma_{f_0}(\bfs_{f_0}^*)\Vert \leq C_2\vert f_1-f_0\vert$, we only need to show 
\begin{equation}\label{eq:A_func}
  \Vert A(f_1,\bfs_{f_1}^*)-A(f_0,\bfs_{f_0}^*)\Vert \leq C_A\vert f_1-f_0\vert,
\end{equation} and 
\begin{equation}\label{eq:B_func}
  \Vert B(f_1,\bfs_{f_1}^*) - B(f_0,\bfs_{f_0}^*)\Vert \leq C_B\vert f_1-f_0\vert.
\end{equation}
In fact, according to the Woodbury matrix identity $$(A-B)^{-1} = A^{-1} +A^{-1}B(A-B)^{-1}$$ that 
\begin{equation}
    \begin{aligned}
        A^{-1}(f_0,\bfs_{f_0}^*) &= (A(f_1,\bfs_{f_1}^*) - [A(f_1,\bfs_{f_1}^*) - A(f_0,\bfs_{f_0}^*)])^{-1}\\
        &= A^{-1}(f_1,\bfs_{f_1}^*) + A^{-1}(f_1,\bfs_{f_1}^*)[A(f_1,\bfs_{f_1}^*) - A(f_0,\bfs_{f_0}^*)]A^{-1}(f_0,\bfs_{f_0}^*)
    \end{aligned}  
\end{equation}
Consequently,
\begin{equation}\label{BA}
    \begin{aligned}
        \Vert A^{-1}(f_1,\bfs_{f_1}^*) - A^{-1}(f_0,\bfs_{f_0}^*)\Vert &\leq \Vert A^{-1}(f_1,\bfs_{f_1}^*)\Vert \Vert A(f_1,\bfs_{f_1}^*) - A(f_0,\bfs_{f_0}^*)\Vert \Vert A^{-1}(f_0,\bfs_{f_0}^*)\Vert\\
        &\leq \lambda_A^2 C_A\vert f_1-f_0\vert.
    \end{aligned}
\end{equation}
By \eqref{sigma_f1}, \eqref{sigma_f0}, \eqref{information}, \eqref{eq:A_func} and \eqref{BA}, we have
\begin{equation}\label{eq:var_target}
    \begin{aligned}
        \Vert &\Sigma_{f_1}(\bfs_{f_1}^*) - \Sigma_{f_0}(\bfs_{f_0}^*)\Vert \\
        =&\Vert A^{-1}(f_1,\bfs_{f_1}^*)B(f_1,\bfs_{f_1}^*)A^{-1}(f_1,\bfs_{f_1}^*) - A^{-1}(f_0,\bfs_{f_0}^*)B(f_0,\bfs_{f_0}^*)A^{-1}(f_0,\bfs_{f_0}^*)\Vert\\
        =& \Vert A^{-1}(f_1,\bfs_{f_1}^*)B(f_1,\bfs_{f_1}^*)A^{-1}(f_1,\bfs_{f_1}^*) - A^{-1}(f_1,\bfs_{f_1}^*)B(f_0,\bfs_{f_0}^*)A^{-1}(f_1,\bfs_{f_1}^*)\\
        & + A^{-1}(f_1,\bfs_{f_1}^*)B(f_0,\bfs_{f_0}^*)A^{-1}(f_1,\bfs_{f_1}^*)- A^{-1}(f_0,\bfs_{f_0}^*)B(f_0,\bfs_{f_0}^*)A^{-1}(f_0,\bfs_{f_0}^*)\Vert\\
        \leq & \Vert A^{-1}(f_1,\bfs_{f_1}^*)\Vert \Vert B(f_1,\bfs_{f_1}^*) - B(f_0,\bfs_{f_0}^*)\Vert \Vert A^{-1}(f_1,\bfs_{f_1}^*)\Vert \\
        &+ \Vert A^{-1}(f_1,\bfs_{f_1}^*) - A^{-1}(f_0,\bfs_{f_0}^*)\Vert \Vert B(f_0,\bfs_{f_0}^*) \Vert \Vert A^{-1}(f_1,\bfs_{f_1}^*)+A^{-1}(f_0,\bfs_{f_0}^*)\Vert\\
        \leq & (\lambda_A^2 C_B + 2\lambda_A^3C_A\Lambda_B)\vert f_1-f_0\vert.
    \end{aligned}
\end{equation}
Here we let $C_2 = \lambda_A^2 C_B + 2\lambda_A^3C_A\Lambda_B$ which is given in Theorem 3 in the main text.

Now we prove Equations \eqref{eq:A_func} and \eqref{eq:B_func}. Given the prevalence constraint \eqref{eq:f_constrain}, i.e., $f=F(\alpha,\bfs)$, we have
$$
    \frac{\partial\alpha(f,\bfs)}{\partial \bfs} = -\frac{\partial F/\partial\bfs}{\partial F/\partial\alpha}(f,\bfs).
$$
It can be verified that when $f\in (0,1-\epsilon]$, $\bfs\in B$, $\partial\alpha(f,\bfs)/\partial\bfs$ is bounded Lipschitz continuous with respect to $f$, and 
the derivative 
$$
    \frac{\partial l_{f}(\bfs)}{\partial \bfs} = 
    \begin{bmatrix}
    (D-\frac{\exp(\alpha+\beta X+\gamma E)}{1+\exp(\alpha+\beta X+\gamma E)})(X + \frac{\partial\alpha}{\partial\beta})\\
    (D - \frac{\exp(\alpha+\beta X+\gamma E)}{1+\exp(\alpha+\beta X+\gamma E)})(E +\frac{\partial\alpha}{\partial\gamma})\\
    \frac{X}{\theta} - \frac{1-X}{1-\theta} + (D-\frac{\exp(\alpha+\beta X+\gamma E)}{1+\exp(\alpha+\beta X+\gamma E)})\frac{\partial\alpha}{\partial\theta}\\
    \frac{E}{\pi} - \frac{1-E}{1-\pi}+(D-\frac{\exp(\alpha+\beta X+\gamma E)}{1+\exp(\alpha+\beta X+\gamma E)})\frac{\partial\alpha}{\partial\pi}
    \end{bmatrix}
$$
is also bounded Lipschitz continuous with respect to $f$, since $\exp(\alpha(f,\bfs))/(1+\exp(\alpha(f,\bfs)))$ is bounded Lipschitz continuous with respect to $f$. By the fact that the product of two bounded Lipschitz continuous functions is also bounded Lipschitz continuous, we have $\{\partial l_f(\bfs) / \partial\bfs\}\{\partial l_f(\bfs) / \partial\bfs^\top\}$ is Lipschitz continuous with respect to $f$ (assume the Lipschitz constant $L_{Bf}$). Moreover, $\{\partial l_f(\bfs) / \partial\bfs\}\{\partial l_f(\bfs) / \partial\bfs^\top\}$ is a continuously differentiable function with respect to $\bfs$ in the compact set $B$, so $B(f,\bfs)$ is Lipschitz continuous with respect to $\bfs$ (assume the Lipschitz constant $L_{Bs}$). Using $\Vert \bfs_{f_1}^*-\bfs_{f_0}^*\Vert \leq C_1\vert f_1-f_0\vert$ proved in Section \ref{sec:7.2}, we have that Equation~\eqref{eq:B_func} holds:
\begin{equation*}
\begin{aligned}
    \Vert B(f_1,\bfs_{f_1}^*) - B(f_0,\bfs_{f_0}^*)\Vert &\leq \Vert B(f_1,\bfs_{f_1}^*) - B(f_0,\bfs_{f_1}^*)\Vert + \Vert B(f_0,\bfs_{f_1}^*) - B(f_0,\bfs_{f_0}^*)\Vert\\
  & \leq (L_{Bf}+ L_{Bs}C_1)\vert f_1-f_0\vert.
\end{aligned}
\end{equation*}
Let $C_B = L_{Bf}+ L_{Bs}C_1$ which is defined in \eqref{eq:B_func}.

Similarly, 
$$
    \frac{\partial^2\alpha(f,\bfs)}{\partial\bfs\partial\bfs^\top} = - \frac{\frac{\partial F}{\partial\alpha}(\frac{\partial\alpha}{\partial\bfs})^2+2\frac{\partial^2 F}{\partial\alpha\partial\bfs}\frac{\partial\alpha}{\partial\bfs}+\frac{\partial^2 F}{\partial\bfs\partial\bfs^\top}}{\frac{\partial F}{\partial\alpha}}
$$
is also bounded Lipschitz continuous with respect to $f$.
Denote 
$$l_1(f,\bfs)=\frac{\partial l(\alpha,\bfs)}{\partial\alpha}, l_2(f,\bfs)=\frac{\partial l(\alpha,\bfs)}{\partial\bfs},$$ 
$$
l_{11}(f,\bfs) = \frac{\partial^2 l(\alpha,\bfs)}{\partial\alpha\partial\alpha},\ l_{12}(f,\bfs) = \frac{\partial^2 l(\alpha,\bfs)}{\partial\alpha\partial\bfs},\ l_{22}(f,\bfs) = \frac{\partial^2 l(\alpha,\bfs)}{\partial\bfs\partial\bfs}.
$$
Since $l_f(\bfs) = l(\alpha(f,\bfs),\bfs)$, we have
\begin{align*}
    \frac{\partial^2 l_f(\bfs)}{\partial\bfs\partial\bfs^\top}&= \frac{\partial}{\partial\bfs}\bigg[l_1(f,\bfs)\frac{\partial\alpha}{\partial\bfs}+l_2(f,\bfs)\bigg]\\
    &=\bigg[l_{11}(f,\bfs)\frac{\partial\alpha}{\partial\bfs}+l_{12}(f,\bfs)\bigg]\frac{\partial\alpha}{\partial\bfs} + l_1(f,\bfs)\frac{\partial^2\alpha}{\partial\bfs\partial\bfs^\top} + l_{12}(f,\bfs)\frac{\partial\alpha}{\partial\bfs}+l_{22}(f,\bfs)\\
    &= l_{11}(f,\bfs)\frac{\partial\alpha}{\partial\bfs}\frac{\partial\alpha}{\partial\bfs^\top} + 2l_{12}(f,\bfs)\frac{\partial\alpha}{\partial\bfs^\top} + l_1(f,\bfs)\frac{\partial^2\alpha}{\partial\bfs\partial\bfs^\top} + l_{22}(f,\bfs)
\end{align*}
is bounded Lipschitz continuous since each of the terms in the right hand side of the above equation is a bounded Lipschitz continuous function with respect to $f$. Also $\partial^2 l_f(\bfs)/(\partial\bfs\partial\bfs^\top)$ is continuously differentiable with respect to $\bfs$ in the compact region $B$, thus $A(f,\bfs)$ is also Lipschitz continuous with respect to $\bfs$. So we have that \eqref{eq:A_func} holds:
\begin{equation*}
\begin{aligned}
  \Vert A(f_1,\bfs_{f_1}^*) - A(f_0,\bfs_{f_0}^*)\Vert &\leq \Vert A(f_1,\bfs_{f_1}^*) - A(f_0,\bfs_{f_1}^*)\Vert + \Vert A(f_0,\bfs_{f_1}^*)-A(f_0,\bfs_{f_0}^*)\Vert\\
  &\leq C_A\vert f_1-f_0\vert.
\end{aligned}
\end{equation*}
Finally, following from \eqref{eq:A_func} and \eqref{eq:B_func}, we have \eqref{eq:var_target} holds.

\section{Additional discussion}\label{sec:additional_discussion}
Adjusting for independent risk factors in randomized clinical trials can help improve estimation efficiency and test power in linear regression analyses \cite[]{fisher1932, kahan2014}. In case-control studies, there is still debate on whether independent covariates should be adjusted for in logistic regression analyses. We theoretically explored three methods's estimation efficiency and power when both the covariate and exposure of interest are binary. Our results can be summarized as follows. First, the estimated odds ratio of the exposure effect with the independent covariate ignored ($\hat{\gamma}_M$) is smaller than that of the covariate-adjusted estimate ($\hat{\gamma}_A$). This provided theoretical justification for the empirical observations in the literature \cite[]{stringer2011}. Second, the variance of $\hat{\gamma}_M$ is smaller than that of $\hat{\gamma}_A$. This extended results in \citet{pirinen2012} where the outcome was rare. Third, the variance of the estimated odds ratio for the covariate-adjusted exposure effect lies between those of \textsc{Mar} and \textsc{Adj} if the covariate-exposure independence is explicitly accommodated in the maximum likelihood estimation (\textsc{AdjCon}). \textsc{AdjCon} is always more powerful than both \textsc{Mar} and \textsc{Adj}, \textsc{Mar} is more powerful than \textsc{Adj} at low outcome prevalence, and \textsc{Adj} is more powerful than \textsc{Mar} when the outcome prevalence is close to 0.5. Last, we show the statistical inference for the \textsc{AdjCon} method is not sensitive to the outcome prevalence misspecification. These results theoretically confirm the empirical findings in \citet{mpMLE2017}.

The main results above provide clear guidelines for choosing an appropriate approach in case-control studies. We suggest use the constraint maximum likelihood method if computational burden is not an issue. The marginal approach is prefered if the outcome prevalence is small, especially when one is interested in screening variables among a large number of potential risk factors (e.g., in genomewide association analysis studies). 

Our theoretical results were developed in a simple situation where the exposures of interest and the covariate were both binary. Further work is warranted to extend the current results to more general situations. For example, the exposure and covariate can be categorical or even continuous, there could be multiple independent covariates, and the sampling of cases and controls could be stratified. 

There are some works related to ours in the literature. Methods have been developed to exploit gene-environment independence and prevalence information in the analysis of case-control data \cite[]{piegorsch1994non, chatterjee2005semiparametric, mukherjee2008exploiting,chenchen2011,clayton2002, qin2014using}, to improve estimation efficiency and test power. \citet{piegorsch1994non} observed improved efficiency for estimating gene-environment interaction effects using case-control data, which was valid when the gene and environmental risk factors were independent in the population, and the outcome was rare. \citet{chatterjee2005semiparametric} extended this method to incorporate covariates and allowed for a stratified sampling in the context of logistic regression models. \citet{mukherjee2008exploiting}  developed an empirical Bayes shrinkage method to relax the gene-environment independence assumption required in \citet{chatterjee2005semiparametric}. \citet{chenchen2011} observed that no power improvement can be achieved by incorporating gene-environment independence if both gene and environmental factors were dichotomous. \citet{qin2014using} developed a rigorous statistical procedure to utilize covariate-specific outcome prevalence in the context of an exponential tilt model. The improvement in statistical efficiency of the method \textsc{AdjCon} is similar in spirit to these methods.

\bibliographystyle{apacite}
\bibliography{ref}

\begin{thebibliography}{}

\bibitem [\protect \citeauthoryear {%
Anderson%
}{%
Anderson%
}{%
{\protect \APACyear {1972}}%
}]{%
anderson1972}
\APACinsertmetastar {%
anderson1972}%
\begin{APACrefauthors}%
Anderson, J\BPBI A.%
\end{APACrefauthors}%
\unskip\
\newblock
\APACrefYearMonthDay{1972}{}{}.
\newblock
{\BBOQ}\APACrefatitle {Separate sample logistic discrimination} {Separate sample logistic discrimination}.{\BBCQ}
\newblock
\APACjournalVolNumPages{Biometrika}{59}{1}{19--35}.
\PrintBackRefs{\CurrentBib}

\bibitem [\protect \citeauthoryear {%
Breslow%
, Robins%
\BCBL {}\ \BBA {} Wellner%
}{%
Breslow%
\ \protect \BOthers {.}}{%
{\protect \APACyear {2000}}%
}]{%
breslow2000}
\APACinsertmetastar {%
breslow2000}%
\begin{APACrefauthors}%
Breslow, N.%
, Robins, J.%
\BCBL {}\ \BBA {} Wellner, J.%
\end{APACrefauthors}%
\unskip\
\newblock
\APACrefYearMonthDay{2000}{}{}.
\newblock
{\BBOQ}\APACrefatitle {On the semi-parametric efficiency of logistic regression under case-control sampling} {On the semi-parametric efficiency of logistic regression under case-control sampling}.{\BBCQ}
\newblock
\APACjournalVolNumPages{Bernoulli}{6}{}{447-455}.
\PrintBackRefs{\CurrentBib}

\bibitem [\protect \citeauthoryear {%
Chatterjee%
\ \BBA {} Carroll%
}{%
Chatterjee%
\ \BBA {} Carroll%
}{%
{\protect \APACyear {2005}}%
}]{%
chatterjee2005semiparametric}
\APACinsertmetastar {%
chatterjee2005semiparametric}%
\begin{APACrefauthors}%
Chatterjee, N.%
\BCBT {}\ \BBA {} Carroll, R\BPBI J.%
\end{APACrefauthors}%
\unskip\
\newblock
\APACrefYearMonthDay{2005}{}{}.
\newblock
{\BBOQ}\APACrefatitle {Semiparametric maximum likelihood estimation exploiting gene-environment independence in case-control studies} {Semiparametric maximum likelihood estimation exploiting gene-environment independence in case-control studies}.{\BBCQ}
\newblock
\APACjournalVolNumPages{Biometrika}{92}{2}{399--418}.
\PrintBackRefs{\CurrentBib}

\bibitem [\protect \citeauthoryear {%
Chen%
\ \BBA {} Chen%
}{%
Chen%
\ \BBA {} Chen%
}{%
{\protect \APACyear {2011}}%
}]{%
chenchen2011}
\APACinsertmetastar {%
chenchen2011}%
\begin{APACrefauthors}%
Chen, H\BPBI Y.%
\BCBT {}\ \BBA {} Chen, J.%
\end{APACrefauthors}%
\unskip\
\newblock
\APACrefYearMonthDay{2011}{}{}.
\newblock
{\BBOQ}\APACrefatitle {On Information coded in gene-environment independence in case-control studies} {On information coded in gene-environment independence in case-control studies}.{\BBCQ}
\newblock
\APACjournalVolNumPages{American Journal of Epidemiology}{174}{6}{736--743}.
\PrintBackRefs{\CurrentBib}

\bibitem [\protect \citeauthoryear {%
Clayton%
}{%
Clayton%
}{%
{\protect \APACyear {2012}}%
}]{%
clayton2002}
\APACinsertmetastar {%
clayton2002}%
\begin{APACrefauthors}%
Clayton, D.%
\end{APACrefauthors}%
\unskip\
\newblock
\APACrefYearMonthDay{2012}{}{}.
\newblock
{\BBOQ}\APACrefatitle {Link functions in multi-locus genetic models: implications for testing, prediction, and interpretation} {Link functions in multi-locus genetic models: implications for testing, prediction, and interpretation}.{\BBCQ}
\newblock
\APACjournalVolNumPages{Genetic Epidemiology}{36}{4}{409–-418}.
\PrintBackRefs{\CurrentBib}

\bibitem [\protect \citeauthoryear {%
Fisher%
}{%
Fisher%
}{%
{\protect \APACyear {1932}}%
}]{%
fisher1932}
\APACinsertmetastar {%
fisher1932}%
\begin{APACrefauthors}%
Fisher, R\BPBI A.%
\end{APACrefauthors}%
\unskip\
\newblock
\APACrefYear{1932}.
\newblock
\APACrefbtitle {Statistical Methods for Research Workers} {Statistical methods for research workers}.
\newblock
\APACaddressPublisher{Edinburgh}{Oliver \& Boyd (13th ed., 1958)}.
\PrintBackRefs{\CurrentBib}

\bibitem [\protect \citeauthoryear {%
Gail%
, Wieand%
\BCBL {}\ \BBA {} Piantadosi%
}{%
Gail%
\ \protect \BOthers {.}}{%
{\protect \APACyear {1984}}%
}]{%
gail1984biased}
\APACinsertmetastar {%
gail1984biased}%
\begin{APACrefauthors}%
Gail, M\BPBI H.%
, Wieand, S.%
\BCBL {}\ \BBA {} Piantadosi, S.%
\end{APACrefauthors}%
\unskip\
\newblock
\APACrefYearMonthDay{1984}{}{}.
\newblock
{\BBOQ}\APACrefatitle {Biased estimates of treatment effect in randomized experiments with nonlinear regressions and omitted covariates} {Biased estimates of treatment effect in randomized experiments with nonlinear regressions and omitted covariates}.{\BBCQ}
\newblock
\APACjournalVolNumPages{Biometrika}{71}{3}{431--444}.
\PrintBackRefs{\CurrentBib}

\bibitem [\protect \citeauthoryear {%
Gart%
}{%
Gart%
}{%
{\protect \APACyear {1962}}%
}]{%
gart1962combination}
\APACinsertmetastar {%
gart1962combination}%
\begin{APACrefauthors}%
Gart, J\BPBI J.%
\end{APACrefauthors}%
\unskip\
\newblock
\APACrefYearMonthDay{1962}{}{}.
\newblock
{\BBOQ}\APACrefatitle {On the combination of relative risks} {On the combination of relative risks}.{\BBCQ}
\newblock
\APACjournalVolNumPages{Biometrics}{18}{4}{601--610}.
\PrintBackRefs{\CurrentBib}

\bibitem [\protect \citeauthoryear {%
Kahan%
, Jairath%
, Dor\'{e}%
\BCBL {}\ \BBA {} Morris%
}{%
Kahan%
\ \protect \BOthers {.}}{%
{\protect \APACyear {2014}}%
}]{%
kahan2014}
\APACinsertmetastar {%
kahan2014}%
\begin{APACrefauthors}%
Kahan, B\BPBI C.%
, Jairath, V.%
, Dor\'{e}, C\BPBI J.%
\BCBL {}\ \BBA {} Morris, T\BPBI P.%
\end{APACrefauthors}%
\unskip\
\newblock
\APACrefYearMonthDay{2014}{}{}.
\newblock
{\BBOQ}\APACrefatitle {The risks and rewards of covariate adjustment in randomized trials: an assessment of 12 outcomes from 8 studies} {The risks and rewards of covariate adjustment in randomized trials: an assessment of 12 outcomes from 8 studies}.{\BBCQ}
\newblock
\APACjournalVolNumPages{Trials}{15}{}{139}.
\PrintBackRefs{\CurrentBib}

\bibitem [\protect \citeauthoryear {%
Kuo%
\ \BBA {} Feingold%
}{%
Kuo%
\ \BBA {} Feingold%
}{%
{\protect \APACyear {2010}}%
}]{%
kuo2010}
\APACinsertmetastar {%
kuo2010}%
\begin{APACrefauthors}%
Kuo, C\BHBI L.%
\BCBT {}\ \BBA {} Feingold, E.%
\end{APACrefauthors}%
\unskip\
\newblock
\APACrefYearMonthDay{2010}{}{}.
\newblock
{\BBOQ}\APACrefatitle {What's the best statistic for a simple test of genetic association in a case-control study?} {What's the best statistic for a simple test of genetic association in a case-control study?}{\BBCQ}
\newblock
\APACjournalVolNumPages{Genetic Epidemiology}{34}{3}{246--253}.
\PrintBackRefs{\CurrentBib}

\bibitem [\protect \citeauthoryear {%
Lee%
}{%
Lee%
}{%
{\protect \APACyear {1982}}%
}]{%
lee1982specification}
\APACinsertmetastar {%
lee1982specification}%
\begin{APACrefauthors}%
Lee, L\BHBI F.%
\end{APACrefauthors}%
\unskip\
\newblock
\APACrefYearMonthDay{1982}{}{}.
\newblock
{\BBOQ}\APACrefatitle {Specification error in multinomial logit models: Analysis of the omitted variable bias} {Specification error in multinomial logit models: Analysis of the omitted variable bias}.{\BBCQ}
\newblock
\APACjournalVolNumPages{Journal of Econometrics}{20}{2}{197--209}.
\PrintBackRefs{\CurrentBib}

\bibitem [\protect \citeauthoryear {%
Mukherjee%
\ \BBA {} Chatterjee%
}{%
Mukherjee%
\ \BBA {} Chatterjee%
}{%
{\protect \APACyear {2008}}%
}]{%
mukherjee2008exploiting}
\APACinsertmetastar {%
mukherjee2008exploiting}%
\begin{APACrefauthors}%
Mukherjee, B.%
\BCBT {}\ \BBA {} Chatterjee, N.%
\end{APACrefauthors}%
\unskip\
\newblock
\APACrefYearMonthDay{2008}{}{}.
\newblock
{\BBOQ}\APACrefatitle {Exploiting gene-environment independence for analysis of case--control studies: an empirical {B}ayes-type shrinkage estimator to trade-off between bias and efficiency} {Exploiting gene-environment independence for analysis of case--control studies: an empirical {B}ayes-type shrinkage estimator to trade-off between bias and efficiency}.{\BBCQ}
\newblock
\APACjournalVolNumPages{Biometrics}{64}{3}{685--694}.
\PrintBackRefs{\CurrentBib}

\bibitem [\protect \citeauthoryear {%
Neuhaus%
}{%
Neuhaus%
}{%
{\protect \APACyear {1998}}%
}]{%
neuhaus1998}
\APACinsertmetastar {%
neuhaus1998}%
\begin{APACrefauthors}%
Neuhaus, J.%
\end{APACrefauthors}%
\unskip\
\newblock
\APACrefYearMonthDay{1998}{}{}.
\newblock
{\BBOQ}\APACrefatitle {Estimation efficiency with omitted covariates in generalized linear models} {Estimation efficiency with omitted covariates in generalized linear models}.{\BBCQ}
\newblock
\APACjournalVolNumPages{Journal of the American Statistical Association}{93}{443}{1124--1129}.
\PrintBackRefs{\CurrentBib}

\bibitem [\protect \citeauthoryear {%
Neuhaus%
\ \BBA {} Jewell%
}{%
Neuhaus%
\ \BBA {} Jewell%
}{%
{\protect \APACyear {1993}}%
}]{%
neuhausjewell1993}
\APACinsertmetastar {%
neuhausjewell1993}%
\begin{APACrefauthors}%
Neuhaus, J.%
\BCBT {}\ \BBA {} Jewell, N.%
\end{APACrefauthors}%
\unskip\
\newblock
\APACrefYearMonthDay{1993}{}{}.
\newblock
{\BBOQ}\APACrefatitle {A geometric approach to assess bias due to omitted covariates in generalized linear-models} {A geometric approach to assess bias due to omitted covariates in generalized linear-models}.{\BBCQ}
\newblock
\APACjournalVolNumPages{Biometrika}{80}{4}{807--815}.
\PrintBackRefs{\CurrentBib}

\bibitem [\protect \citeauthoryear {%
Piegorsch%
, Weinberg%
\BCBL {}\ \BBA {} Taylor%
}{%
Piegorsch%
\ \protect \BOthers {.}}{%
{\protect \APACyear {1994}}%
}]{%
piegorsch1994non}
\APACinsertmetastar {%
piegorsch1994non}%
\begin{APACrefauthors}%
Piegorsch, W\BPBI W.%
, Weinberg, C\BPBI R.%
\BCBL {}\ \BBA {} Taylor, J\BPBI A.%
\end{APACrefauthors}%
\unskip\
\newblock
\APACrefYearMonthDay{1994}{}{}.
\newblock
{\BBOQ}\APACrefatitle {Non-hierarchical logistic models and case-only designs for assessing susceptibility in population-based case-control studies} {Non-hierarchical logistic models and case-only designs for assessing susceptibility in population-based case-control studies}.{\BBCQ}
\newblock
\APACjournalVolNumPages{Statistics in Medicine}{13}{2}{153--162}.
\PrintBackRefs{\CurrentBib}

\bibitem [\protect \citeauthoryear {%
Pirinen%
, Donnelly%
\BCBL {}\ \BBA {} Spencer%
}{%
Pirinen%
\ \protect \BOthers {.}}{%
{\protect \APACyear {2012}}%
}]{%
pirinen2012}
\APACinsertmetastar {%
pirinen2012}%
\begin{APACrefauthors}%
Pirinen, M.%
, Donnelly, P.%
\BCBL {}\ \BBA {} Spencer, C\BPBI C\BPBI A.%
\end{APACrefauthors}%
\unskip\
\newblock
\APACrefYearMonthDay{2012}{}{}.
\newblock
{\BBOQ}\APACrefatitle {Including known covariates can reduce power to detect genetic effects in case-control studies} {Including known covariates can reduce power to detect genetic effects in case-control studies}.{\BBCQ}
\newblock
\APACjournalVolNumPages{Nature Genetics}{44}{8}{848--851}.
\PrintBackRefs{\CurrentBib}

\bibitem [\protect \citeauthoryear {%
Pitman%
}{%
Pitman%
}{%
{\protect \APACyear {1979}}%
}]{%
pitman1979}
\APACinsertmetastar {%
pitman1979}%
\begin{APACrefauthors}%
Pitman, E\BPBI J\BPBI G.%
\end{APACrefauthors}%
\unskip\
\newblock
\APACrefYear{1979}.
\newblock
\APACrefbtitle {Some Basic Theory of Statistical Inference} {Some basic theory of statistical inference}.
\newblock
\APACaddressPublisher{London}{Chapman \& Hall}.
\PrintBackRefs{\CurrentBib}

\bibitem [\protect \citeauthoryear {%
Prentice%
\ \BBA {} Pyke%
}{%
Prentice%
\ \BBA {} Pyke%
}{%
{\protect \APACyear {1979}}%
}]{%
prentice1979logistic}
\APACinsertmetastar {%
prentice1979logistic}%
\begin{APACrefauthors}%
Prentice, R\BPBI L.%
\BCBT {}\ \BBA {} Pyke, R.%
\end{APACrefauthors}%
\unskip\
\newblock
\APACrefYearMonthDay{1979}{}{}.
\newblock
{\BBOQ}\APACrefatitle {Logistic disease incidence models and case-control studies} {Logistic disease incidence models and case-control studies}.{\BBCQ}
\newblock
\APACjournalVolNumPages{Biometrika}{66}{3}{403--411}.
\PrintBackRefs{\CurrentBib}

\bibitem [\protect \citeauthoryear {%
Qin%
, Zhang%
, Li%
, Albanes%
\BCBL {}\ \BBA {} Yu%
}{%
Qin%
\ \protect \BOthers {.}}{%
{\protect \APACyear {2014}}%
}]{%
qin2014using}
\APACinsertmetastar {%
qin2014using}%
\begin{APACrefauthors}%
Qin, J.%
, Zhang, H.%
, Li, P.%
, Albanes, D.%
\BCBL {}\ \BBA {} Yu, K.%
\end{APACrefauthors}%
\unskip\
\newblock
\APACrefYearMonthDay{2014}{}{}.
\newblock
{\BBOQ}\APACrefatitle {Using covariate-specific disease prevalence information to increase the power of case-control studies} {Using covariate-specific disease prevalence information to increase the power of case-control studies}.{\BBCQ}
\newblock
\APACjournalVolNumPages{Biometrika}{102}{1}{169--180}.
\PrintBackRefs{\CurrentBib}

\bibitem [\protect \citeauthoryear {%
Robinson%
\ \BBA {} Jewell%
}{%
Robinson%
\ \BBA {} Jewell%
}{%
{\protect \APACyear {1991}}%
}]{%
robinson1991some}
\APACinsertmetastar {%
robinson1991some}%
\begin{APACrefauthors}%
Robinson, L\BPBI D.%
\BCBT {}\ \BBA {} Jewell, N\BPBI P.%
\end{APACrefauthors}%
\unskip\
\newblock
\APACrefYearMonthDay{1991}{}{}.
\newblock
{\BBOQ}\APACrefatitle {Some surprising results about covariate adjustment in logistic regression models} {Some surprising results about covariate adjustment in logistic regression models}.{\BBCQ}
\newblock
\APACjournalVolNumPages{International Statistical Review/Revue Internationale de Statistique}{59}{2}{227--240}.
\PrintBackRefs{\CurrentBib}

\bibitem [\protect \citeauthoryear {%
Serfling%
}{%
Serfling%
}{%
{\protect \APACyear {2009}}%
}]{%
serfling2009}
\APACinsertmetastar {%
serfling2009}%
\begin{APACrefauthors}%
Serfling, R\BPBI J.%
\end{APACrefauthors}%
\unskip\
\newblock
\APACrefYear{2009}.
\newblock
\APACrefbtitle {Approximation Theorems of Mathematical Statistics} {Approximation theorems of mathematical statistics}\ (\BVOL~162).
\newblock
\APACaddressPublisher{}{John Wiley \& Sons}.
\PrintBackRefs{\CurrentBib}

\bibitem [\protect \citeauthoryear {%
Stringer%
, Wray%
, Kahn%
\BCBL {}\ \BBA {} Derks%
}{%
Stringer%
\ \protect \BOthers {.}}{%
{\protect \APACyear {2011}}%
}]{%
stringer2011}
\APACinsertmetastar {%
stringer2011}%
\begin{APACrefauthors}%
Stringer, S.%
, Wray, N\BPBI R.%
, Kahn, R\BPBI S.%
\BCBL {}\ \BBA {} Derks, E\BPBI M.%
\end{APACrefauthors}%
\unskip\
\newblock
\APACrefYearMonthDay{2011}{}{}.
\newblock
{\BBOQ}\APACrefatitle {Underestimated effect sizes in {GWAS}: fundamental limitations of single {SNP} analysis for dichotomous phenotypes} {Underestimated effect sizes in {GWAS}: fundamental limitations of single {SNP} analysis for dichotomous phenotypes}.{\BBCQ}
\newblock
\APACjournalVolNumPages{PLoS ONE}{6}{11}{e27964}.
\PrintBackRefs{\CurrentBib}

\bibitem [\protect \citeauthoryear {%
Van~der Vaart%
}{%
Van~der Vaart%
}{%
{\protect \APACyear {2000}}%
}]{%
van2000asymptotic}
\APACinsertmetastar {%
van2000asymptotic}%
\begin{APACrefauthors}%
Van~der Vaart, A\BPBI W.%
\end{APACrefauthors}%
\unskip\
\newblock
\APACrefYear{2000}.
\newblock
\APACrefbtitle {Asymptotic statistics} {Asymptotic statistics}\ (\BVOL~3).
\newblock
\APACaddressPublisher{}{Cambridge University Press}.
\PrintBackRefs{\CurrentBib}

\bibitem [\protect \citeauthoryear {%
White%
}{%
White%
}{%
{\protect \APACyear {1982}}%
}]{%
white1982maximum}
\APACinsertmetastar {%
white1982maximum}%
\begin{APACrefauthors}%
White, H.%
\end{APACrefauthors}%
\unskip\
\newblock
\APACrefYearMonthDay{1982}{}{}.
\newblock
{\BBOQ}\APACrefatitle {Maximum likelihood estimation of misspecified models} {Maximum likelihood estimation of misspecified models}.{\BBCQ}
\newblock
\APACjournalVolNumPages{Econometrica: Journal of the Econometric Society}{}{}{1--25}.
\PrintBackRefs{\CurrentBib}

\bibitem [\protect \citeauthoryear {%
Xing%
\ \BBA {} Xing%
}{%
Xing%
\ \BBA {} Xing%
}{%
{\protect \APACyear {2010}}%
}]{%
xing2010adjusting}
\APACinsertmetastar {%
xing2010adjusting}%
\begin{APACrefauthors}%
Xing, G.%
\BCBT {}\ \BBA {} Xing, C.%
\end{APACrefauthors}%
\unskip\
\newblock
\APACrefYearMonthDay{2010}{}{}.
\newblock
{\BBOQ}\APACrefatitle {Adjusting for covariates in logistic regression models} {Adjusting for covariates in logistic regression models}.{\BBCQ}
\newblock
\APACjournalVolNumPages{Genetic Epidemiology}{34}{7}{769}.
\PrintBackRefs{\CurrentBib}

\bibitem [\protect \citeauthoryear {%
Zaitlen%
, Lindstr{\"o}m%
\BCBL {}\ \protect \BOthers {.}}{%
Zaitlen%
, Lindstr{\"o}m%
\BCBL {}\ \protect \BOthers {.}}{%
{\protect \APACyear {2012}}%
}]{%
zaitlen2012informed}
\APACinsertmetastar {%
zaitlen2012informed}%
\begin{APACrefauthors}%
Zaitlen, N.%
, Lindstr{\"o}m, S.%
, Pasaniuc, B.%
, Cornelis, M.%
, Genovese, G.%
, Pollack, S.%
\BDBL {}Price, A\BPBI L.%
\end{APACrefauthors}%
\unskip\
\newblock
\APACrefYearMonthDay{2012}{}{}.
\newblock
{\BBOQ}\APACrefatitle {Informed conditioning on clinical covariates increases power in case-control association studies} {Informed conditioning on clinical covariates increases power in case-control association studies}.{\BBCQ}
\newblock
\APACjournalVolNumPages{PLoS Genetics}{8}{11}{e1003032}.
\PrintBackRefs{\CurrentBib}

\bibitem [\protect \citeauthoryear {%
Zaitlen%
, Pa{\c{s}}aniuc%
\BCBL {}\ \protect \BOthers {.}}{%
Zaitlen%
, Pa{\c{s}}aniuc%
\BCBL {}\ \protect \BOthers {.}}{%
{\protect \APACyear {2012}}%
}]{%
zaitlen2012analysis}
\APACinsertmetastar {%
zaitlen2012analysis}%
\begin{APACrefauthors}%
Zaitlen, N.%
, Pa{\c{s}}aniuc, B.%
, Patterson, N.%
, Pollack, S.%
, Voight, B.%
, Groop, L.%
\BDBL {}Price, A\BPBI L.%
\end{APACrefauthors}%
\unskip\
\newblock
\APACrefYearMonthDay{2012}{}{}.
\newblock
{\BBOQ}\APACrefatitle {Analysis of case-control association studies with known risk variants} {Analysis of case-control association studies with known risk variants}.{\BBCQ}
\newblock
\APACjournalVolNumPages{Bioinformatics}{28}{13}{1729--1737}.
\PrintBackRefs{\CurrentBib}

\bibitem [\protect \citeauthoryear {%
Zhang%
, Chatterjee%
, Rader%
\BCBL {}\ \BBA {} Chen%
}{%
Zhang%
\ \protect \BOthers {.}}{%
{\protect \APACyear {2018}}%
}]{%
mpMLE2017}
\APACinsertmetastar {%
mpMLE2017}%
\begin{APACrefauthors}%
Zhang, H.%
, Chatterjee, N.%
, Rader, D.%
\BCBL {}\ \BBA {} Chen, J.%
\end{APACrefauthors}%
\unskip\
\newblock
\APACrefYearMonthDay{2018}{}{}.
\newblock
{\BBOQ}\APACrefatitle {Adjustment of non-confounding covariates in case-control genetic association studies} {Adjustment of non-confounding covariates in case-control genetic association studies}.{\BBCQ}
\newblock
\APACjournalVolNumPages{Annals of Applied Statistics}{12}{}{200--221}.
\PrintBackRefs{\CurrentBib}

\end{thebibliography}

\end{document}